\documentclass[12pt]{article}

\usepackage{amssymb,amsmath,amsthm,thmtools,mathtools,tikz}
\usepackage[hidelinks]{hyperref}
\usepackage[nameinlink]{cleveref}
\newtheorem{theorem}{Theorem}[section]
\newtheorem{corollary}[theorem]{Corollary}

\newtheorem{conjecture}[theorem]{Conjecture}
\newtheorem{lemma}[theorem]{Lemma}

\newcommand{\x}{{\bf x}}
\newcommand{\y}{{\bf y}}
\newcommand{\z}{{\bf z}}
\newcommand{\bu}{{\bf u}}
\newcommand{\bv}{{\bf v}}
\newcommand{\1}{{\bf 1}}
\newcommand{\0}{{\bf 0}}
\newcommand{\al}{\alpha}
\newcommand{\dm}{{\rm diam}}

\newcommand{\e}{\epsilon}
\newcommand{\E}{\mathcal E}
\tikzstyle{vertex}=[circle, draw, inner sep=0pt, minimum size=3pt]
\newcommand{\vertex}{\node[vertex]}

\begin{document}
\title{Graphs with Minimum Algebraic Connectivity I:  \\ Proofs of
	Aldous--Fill and Guiduli--Mohar Conjectures}
\author{ Maryam Abdi$^{\,\rm a}$ \quad  Ebrahim Ghorbani$^{\,\rm b}$
	\\[.3cm]
	{\sl\normalsize $^{\rm a}$School of Mathematics, Institute for Research in Fundamental Sciences (IPM),}\\
	{\sl\normalsize P. O. Box 19395-5746, Tehran, Iran }\\
	{\sl\normalsize $^{\rm b}$Hamburg University of Technology, Institute for Algorithms and Complexity, Germany} \\
	{\tt\small m.abdi@ipm.ir\qquad  ebrahim.ghorbani@tuhh.de} }
\date{}
\maketitle
\begin{abstract}
	Aldous and Fill (2002) conjectured that the maximum relaxation time of a random walk on a connected regular graph with $n$ vertices is bounded above by $(1+o(1))\frac{3n^2}{2\pi^2}$, with asymptotic equality for even $n$. Since the relaxation time of a $d$-regular graph $G$ is $d/\mu(G)$, where $\mu(G)$ denotes its algebraic connectivity, this conjecture is closely related to the problem of minimizing algebraic connectivity among regular graphs.	
	Guiduli and Mohar (1996) conjectured that, for every fixed minimum degree $\delta=d\ge 3$ and all sufficiently large orders, graphs with minimum algebraic connectivity are path-like and, apart from bounded portions near their two ends, have a prescribed block structure. For fixed odd degree $d\ge 3$, Abdi and Ghorbani (2004) conjectured that $d$-regular graphs with minimum algebraic connectivity have the same structure.	
	We prove the Aldous--Fill conjecture and the Guiduli--Mohar conjecture, as well as the corresponding conjecture for $d$-regular graphs of fixed odd degree. Finally, we prove that, for every fixed odd degree $d\ge 3$, $d$-regular graphs, as well as graphs of fixed minimum degree $d\ge 3$, whose algebraic connectivity is asymptotically minimum have asymptotically maximum diameter. This establishes the corresponding cases of another conjecture of Abdi and Ghorbani.

\vspace{4mm}

\noindent {\bf AMS Mathematics Subject Classification\,(2020):}
05C50, 05C35
\end{abstract}

\section{Introduction}\label{sec:intro}

All graphs we consider are simple, i.e. undirected graphs without loops or multiple edges.
Additionally, we assume that they are connected.
The {\em relaxation time} of the random walk on a graph $G$  is defined by $\tau=1/(1-\eta_2)$, where $\eta_2$ is the second largest eigenvalue of
the {\em transition matrix} of $G$, that is, the matrix $\Delta^{-1} A$ in which $\Delta$ and $A$ are the diagonal matrix of vertex degrees and the adjacency matrix of $G$, respectively. If $G$ is $d$-regular, then $\tau=d/\mu(G)$, where $\mu(G)$ denotes the {\em algebraic connectivity} of $G$, namely, the second smallest eigenvalue of its Laplacian matrix $L(G)=\Delta-A$.
For the mixing-time discussion, we use the continuous-time walk with
generator $\Delta^{-1}A-I$; the discrete-time walk may be periodic.
A central problem in the study of random walks is to determine the {\em mixing time}, a measure of how fast the
random walk converges to the stationary distribution. As seen in the literature \cite{AldousFill}, the  relaxation time is the primary term controlling mixing time. Therefore, relaxation time is directly associated with the rate of convergence  of the random walk.
Our main motivation in this work is the following conjecture on the maximum relaxation time of the random walk in regular graphs. 
\begin{conjecture}[{Aldous and Fill \cite[p.~217]{AldousFill}}]\label{conj:A-F} \rm
	Over all regular graphs on $n$ vertices,
	$\max \tau =\max \frac{d}{\mu} \le(1+o(1)) \frac{3n^2}{2\pi^2}$,
	with asymptotic equality for even $n$.
\end{conjecture}

It is worth mentioning that in \cite{actt}, it is proved that the maximum relaxation time for the random walk on a   graph on $n$ vertices is $(1 +o(1))\frac{n^3}{54}$, settling another conjecture by Aldous and Fill \cite[p.~216]{AldousFill}.

\subsection{Graphs minimizing algebraic connectivity}

 We say that a graph $G$ is a {\em $\mu$-minimal} graph
with $\delta=d$ if it has the smallest $\mu$ among all graphs of the
same order with minimum degree $d$. We use the same terminology for
the class $\delta\ge d$ and for the class of $d$-regular graphs.

L.  Babai (see \cite{Guiduli}) made a conjecture that described the structure of  $\mu$-minimal cubic (i.e. $3$-regular) graphs. Guiduli \cite{Guiduli} (see also \cite{GuiduliThesis}) proved that $\mu$-minimal cubic graphs are path-like, built from specific blocks.
The result of Guiduli was improved  later  by Brand, Guiduli, and Imrich \cite{Imrich}. They completely characterized $\mu$-minimal cubic  graphs and confirmed the Babai conjecture. For every admissible even order, the minimizing cubic graph is unique. (Cubic graphs always have even orders.)
Abdi, Ghorbani and Imrich~\cite{AbGhIm} showed that the algebraic connectivity  of these graphs is $(1+o(1))\frac{2\pi^2}{n^2}$, confirming  the Aldous--Fill conjecture for $d=3$.
Guiduli \cite[Problem~5.2]{GuiduliThesis} asked for a generalization of the aforementioned result of Brand, Guiduli, and Imrich, namely the characterization of $\mu$-minimal $d$-regular graphs. 
In this direction, Abdi and Ghorbani \cite{AbGhQuartic} gave a nearly complete characterization
of $\mu$-minimal quartic  (i.e. $4$-regular) graphs.  Based on that,  they established   the Aldous--Fill conjecture for $d=4$. 
Guiduli and Mohar proposed the following generalization of the cubic
result by considering graphs with minimum degree $d$ rather than
regular graphs.
Put $L_d=K_{d+1}-e$. An $L_d$-chain consists of pairwise disjoint copies of $L_d$ arranged in a path-like structure, with consecutive copies joined by bridges; see \Cref{fig:Mohar}.


\begin{figure}[t]
	\centering
	\begin{tikzpicture}[scale=.9]
		\draw (1,0) ellipse (.25 and 1.28);
		\vertex[fill] (0) at (-.3,0) [] {};
		\vertex[fill] (1) at (1,1) [] {};
		\vertex[fill] (2) at (1,.6) [] {};
		\vertex[fill] (3) at (1,-1) [] {};
		\vertex[fill] (4) at (2.3,0) [] {};
		\vertex[fill] (5) at (3.1,0) [] {};
		\draw (4.4,0) ellipse (.25 and 1.28);
		\vertex[fill] (6) at (4.4,1) [] {};
		\vertex[fill] (7) at (4.4,.6) [] {};
		\vertex[fill] (8) at (4.4,-1) [] {};
		\vertex[fill] (9) at (5.6,0) [] {};
		\vertex[fill] (12) at (7.4,0) [] {};
		\draw (8.7,0) ellipse (.25 and 1.28);
		\vertex[fill] (13) at (8.7,1) [] {};
		\vertex[fill] (14) at (8.7,.6) [] {};
		\vertex[fill] (15) at (8.7,-1) [] {};
		\vertex[fill] (16) at (10,0) [] {};
		\vertex[fill] (17) at (10.8,0) [] {};
		\draw (12.1,0) ellipse (.25 and 1.28);
		\vertex[fill] (18) at (12.1,1) [] {};
		\vertex[fill] (19) at (12.1,.6) [] {};
		\vertex[fill] (20) at (12.1,-1) [] {};
		\vertex[fill] (00) at (13.4,0) [] {};
		\tikzstyle{vertex}=[circle, draw, inner sep=.3pt, minimum size=.3pt]
		\vertex[fill] () at (1,-.1) [] {};
		\vertex[fill] () at (1,-.2) [] {};
		\vertex[fill] () at (1,-.3) [] {};
		\vertex[fill] () at (4.4,-.1) [] {};
		\vertex[fill] () at (4.4,-.2) [] {};
		\vertex[fill] () at (4.4,-.3) [] {};
		\vertex[fill] () at (6.4,0) [] {};
		\vertex[fill] () at (6.5,0) [] {};
		\vertex[fill] () at (6.6,0) [] {};
		\vertex[fill] () at (8.7,-.1) [] {};
		\vertex[fill] () at (8.7,-.2) [] {};
		\vertex[fill] () at (8.7,-.3) [] {};
		\vertex[fill] () at (12.1,-.1) [ ] {};
		\vertex[fill] () at (12.1,-.2) [ ] {};
		\vertex[fill] () at (12.1,-.3) [] {};
		\tikzstyle{vertex}=[circle, draw, inner sep=0pt, minimum size=0pt]
		\vertex[] (s) at (-.7,0) [label=left:\footnotesize{}] {};
		\vertex[] (ss) at (13.8,0) [label=left:\footnotesize{}] {};
		\vertex[] () at (1,1.3) [label=above:\footnotesize{$K_{d-1}$}] {};
		\vertex[] () at (4.4,1.3) [label=above:\footnotesize{$K_{d-1}$}] {};
		\vertex[] (10) at (6.0,0) [label=left:\footnotesize{}] {};
		\vertex[] (11) at (7,0) [label=left:\footnotesize{}] {};
		\vertex[] () at (8.7,1.3) [label=above:\footnotesize{$K_{d-1}$}] {};
		\vertex[] () at (12.1,1.3) [label=above:\footnotesize{$K_{d-1}$}] {};
		\path
		(1) edge (0)
		(2) edge (0)
		(3) edge (0)
		(1) edge (4)
		(2) edge (4)
		(3) edge (4)
		(4) edge (5)
		(6) edge (5)
		(7) edge (5)
		(8) edge (5)
		(6) edge (9)
		(7) edge (9)
		(8) edge (9)
		(9) edge (10)
		(11) edge (12)
		(12) edge (13)
		(12) edge (14)
		(12) edge (15)
		(16) edge (13)
		(16) edge (14)
		(16) edge (15)
		(16) edge (17)
		(17) edge (18)
		(17) edge (19)
		(17) edge (20) 
		(18) edge (00)
		(19) edge (00)
		(20) edge (00)
		(s) edge (0)
		(ss) edge (00);
	\end{tikzpicture}
	\caption{Structure of $\mu$-minimal graphs with $\delta=d$, excluding at most a constant number of vertices at either end.  Here $K_l$ is the complete graph of order $l$.}\label{fig:Mohar}
\end{figure}
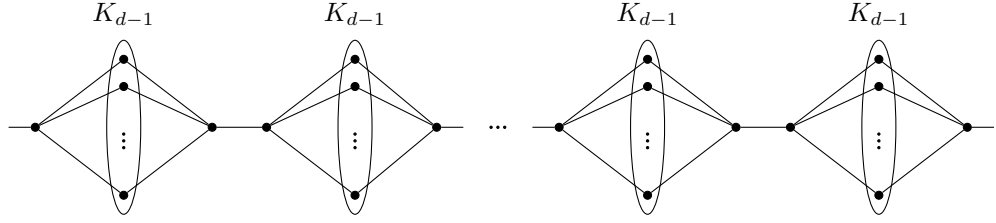

\begin{conjecture}[Guiduli and Mohar, see {\cite[p.~88]{GuiduliThesis}}]\label{conjecture:mohar}
\rm Let $G$ be a $\mu$-minimal graph with $\delta=d$. Then $G$ is
path-like, and except for some blocks near each end, its remaining
blocks form an $L_d$ chain (see \Cref{fig:Mohar}).
\end{conjecture}
Abdi and Ghorbani \cite{AbGhDiam} conjectured that, for odd $d$, the $d$-regular graphs with minimum algebraic connectivity have the same structure as in \Cref{conjecture:mohar}. For even $d$, however, a $d$-regular graph cannot contain a bridge, so this structure is not possible; they therefore conjectured a different  structure for $\mu$-minimal graphs in the even-degree case.


\subsection{The main results}
Our first result proves the Guiduli--Mohar conjecture, \Cref{conjecture:mohar}, together with the odd-degree case of the Abdi--Ghorbani conjecture \cite{AbGhDiam}. It also shows that the $\mu$-minimal graphs in the respective families have asymptotically maximum diameter, thereby settling another conjecture from \cite{AbGhDiam}.

describes the graphs with minimum algebraic
connectivity. It treats the two minimum-degree classes and the
regular class when the degree is odd. For a chosen class, $m_d(n)$
denotes its minimum algebraic connectivity at order $n$. In the
odd-regular class, all orders are even.

\begin{theorem}\label{thm:GMfull}\label{thm:oddregularfull}
For each fixed $d\ge3$, there are constants $N_d,B_d$ such that the
following holds. Let $G$ be a $\mu$-minimal graph of order $n\ge N_d$
in either class $\delta=d$ or $\delta\ge d$. 
Then $G$ is path-like. Outside two connected portions at the ends
of its block-tree, containing at most $B_d$ vertices altogether,
it is a $L_d$ chain. 
When $d$ is odd, the
same conclusion holds if $G$ is $\mu$-minimal among $d$-regular
graphs of order $n$.
\end{theorem}

The companion paper~\cite{AbGhCompanion-even} settles the even-degree case of the conjecture in \cite{AbGhDiam} for $\mu$-minimal regular graphs.

Applying the method of \cite[Section~3.2]{AbGhDiam} for computing the algebraic connectivity of path-like graphs, and accounting for the bounded end portions, yields the following asymptotic formula for the minimum algebraic connectivity.

\begin{corollary}\label{thm:minimum}
For each fixed $d\ge3$, the minimum of $\mu$ at order $n$ in either
class $\delta=d$ or $\delta\ge d$ is
\begin{equation}\label{eq:sharpminimum}
\frac{(d-1)\pi^2}{n^2}+O_d(n^{-3}).
\end{equation}
The same formula holds in the $d$-regular class for each fixed odd
$d\ge3$, with $n$ tending to infinity through even integers.
\end{corollary}

We remark that, for fixed $d$, the maximum diameter in each of the classes $\delta=d$, $\delta\ge d$, and $d$-regular graphs is 
$3n/(d+1)+O_d(1)$; see \cite{Caccetta,Erdos} and \cite[Theorems~5.1--5.2]{AbGhDiam}.
 In \cite[Conjecture~5.4]{AbGhDiam}, it is conjectured that a graph with asymptotically minimum algebraic connectivity has asymptotically maximum diameter within its respective family. The following result settles this conjecture for the families $\delta=d$ and $\delta\ge d$, and, when $d$ is odd, for the family of $d$-regular graphs.
 
\begin{theorem}\label{cor:diameter}
Fix $d\ge3$ and choose the class $\delta=d$ or $\delta\ge d$;
alternatively, fix an odd $d\ge3$ and choose the $d$-regular class.
Let $G_n$ be a graph of order $n$ in the chosen class. If
$\mu(G_n)/m_d(n)\to1$, then
\begin{equation}\label{eq:diameterimplication}
 \dm(G_n)=\frac{3n}{d+1}+o(n).
\end{equation}
For $\mu$-minimal graph $G$ of order $n$ in the chosen class,
\begin{equation}\label{eq:exactdiameter}
 \dm(G)=\frac{3n}{d+1}+O_d(1).
\end{equation}
\end{theorem}

We next allow the minimum degree to vary with the order. The following
bound is uniform over the entire minimum-degree class, without a
regularity assumption.

\begin{theorem}\label{thm:uniformminimum}
For every fixed integer $r\ge3$, put
\[
 c_r=\min\left\{\frac{r-1}{r}\pi^2,\,8\right\}.
\]
Then
\begin{equation}\label{eq:uniformminimum}
 \lim_{n\to\infty}
 \min_{\substack{|V(G)|=n\\ \delta(G)\ge r}}
 \frac{n^2\mu(G)}{\delta(G)}=c_r,
\end{equation}
where the minimum is over connected  graphs.
\end{theorem}

Thus $c_3=2\pi^2/3$, $c_4=3\pi^2/4$, $c_5=4\pi^2/5$, and
$c_r=8$ for $r\ge6$. The matching examples exist at every
sufficiently large order in the minimum-degree class. For irregular
graphs, the ratio $\mu(G)/\delta(G)$ in this theorem is not the
random-walk spectral gap. For a $d$-regular graph it is, and the
case $r=3$, together with the elementary cycle calculation for $d=2$,
proves the Aldous--Fill conjecture, \Cref{conj:A-F}.

\begin{corollary}\label{thm:aldousfill}
For every $\varepsilon>0$ there is $N=N(\varepsilon)$ such that every
connected  $d$-regular graph $G$ of order $n\ge N$, where
$2\le d\le n-1$, satisfies
\[
 \frac{\mu(G)}d\ge(1-\varepsilon)\frac{2\pi^2}{3n^2}.
\]
The constant is sharp along even orders, where cubic minimizers
attain asymptotic equality.
\end{corollary}

We note that recently an independent proof of the Aldous--Fill conjecture
has also been announced by Zhu~\cite{Zhu}. 

\subsection{Outline of the proofs}\label{sec:proofsketch}

Our proofs build on the methods of
\cite{AbGhMin,AbGhDiam,AbGhQuartic,AbGhIm}. The starting point is a
way to express algebraic connectivity using values assigned to the
vertices. Among all real vectors $\x=(x_v)_{v\in V(G)}$ satisfying
$\sum_v x_v=0$ and $\sum_v x_v^2=1$, the minimum of
\[
 \sum_{uv\in E(G)}(x_u-x_v)^2
\]
is $\mu(G)$. A vector attaining this minimum is called a Fiedler
vector. Thus small algebraic connectivity means that these values
can vary across the graph while the total squared difference along
its edges remains small. This also gives a way to rule out a proposed
minimizer: change the graph and find values for which the same sum
is smaller, subject to the same normalization.

This comparison is central to the earlier work.
In \cite{AbGhIm}, edges are switched while preserving every vertex
degree, and the Fiedler vector is used to compare the graphs before
and after the switch. In \cite{AbGhQuartic}, parts of a quartic graph
are replaced and suitable values are assigned to the new vertices
to rule out blocks that cannot occur in a minimizer.
In \cite{AbGhMin}, related comparisons are used for graphs of minimum
degree at least three. The method of \cite[Section~3.2]{AbGhDiam}
then computes the asymptotic algebraic connectivity of the resulting
chains by comparison with a path. We extend these ideas to every
fixed minimum degree $d\ge3$ and to every fixed odd regular degree.

We first arrange the vertices in increasing order of their Fiedler
values. At each position, we count the edges joining the vertices
before that position to those after it. The degree condition gives
lower bounds on these counts. We use them to compare the displayed
sum with the corresponding sum on a path, after replacing some
short sequences of values by equally spaced values. We also keep
track of the nonnegative terms left over in this comparison.
Graphs made from $L_d$ chains with a bounded number of vertices at
their ends give an upper bound of
$(d-1)\pi^2/n^2+O_d(n^{-3})$. A minimizing graph satisfies this
bound too, so the total of the terms left over must be small.
Together with the known Fiedler vector of a path, this shows that
all but $O_d(n^{2/3})$ vertices lie in copies of $L_d$ connected
along a sequence of bridges. More generally, we bound the number
of remaining vertices in terms of how close $\mu(G)$ is to
$(d-1)\pi^2/n^2$. This is the content of \Cref{thm:stability}.

For an exact minimizer, we next examine the connected portions
between these chains and at their ends. If such a portion has too
many vertices, replacing it by a graph of the same order built
mainly from $L_d$ blocks decreases algebraic connectivity.
The comparison accounts for both the edge sum and the normalization
of the new vector. This bounds the size of each remaining portion
by a constant depending only on $d$. If one of these portions lies
far from both ends, exchanging it with one of its neighbouring
$L_d$ blocks again decreases algebraic connectivity. Finally, edge
changes rule out branching, so the blocks occur in a single linear
order. All changes preserve the required degree condition; in the
regular case, every vertex still has degree $d$.
This proves \Cref{thm:GMfull}.

Once this structure is known, we apply
\cite[Section~3.2]{AbGhDiam}. We choose one Fiedler value from each
$L_d$ block and minimize the edge sum between consecutive chosen
vertices. The resulting expression is a constant multiple of the
corresponding sum on a path. The end portions contribute only a
small error, which we estimate to obtain
$\mu(G)=(d-1)\pi^2/n^2+O_d(n^{-3})$.
The diameter conclusion also follows from the arrangement of the
blocks: crossing an $L_d$ block requires at least two edges, and
reaching the next block requires another bridge. The same count,
using \Cref{thm:stability}, proves the conclusion for graphs whose
algebraic connectivity is asymptotically minimum, as conjectured
in \cite[Conjecture~5.4]{AbGhDiam}.

To prove the Aldous--Fill bound, we must also allow the degree to
grow with $n$. We treat two cases. When the minimum degree $d$
satisfies $d/n\to0$, extending the comparison path by $d$ repeated
endpoint values at each end gives a lower bound valid for varying
$d$. When $d$ is at least a fixed positive proportion of $n$, we
retain the edges whose endpoint Fiedler values differ by at most a
small threshold. For the graphs with small algebraic connectivity
that need to be considered, the retained graph has only a bounded
number of connected parts, and the values within each part are
nearly equal. We replace each part by its average value, weighted
by its number of vertices. An inequality comparing these averages
along edges between the parts gives
$n^2\mu(G)/d\ge8-o(1)$. Combining the two cases proves
\Cref{thm:uniformminimum}. For a regular graph, the identity
$\tau=d/\mu(G)$ then gives the Aldous--Fill bound.

\subsection{Organization}

\Cref{sec:prelim} gives estimates for Fiedler vectors and paths.
\Cref{sec:numerical} proves the local inequalities and constructs
comparison graphs. \Cref{sec:stability} uses these facts to prove
\Cref{thm:stability}. \Cref{sec:GMfull} bounds and locates the
exceptional portions of exact minimizers and excludes branching,
proving \Cref{thm:GMfull} in all the stated classes, with separate
checks of the permitted graph changes. \Cref{sec:chainvalue}
deduces \Cref{thm:minimum} from the chain structure, and
\Cref{sec:diameter} proves \Cref{cor:diameter}.
Finally, \Cref{sec:aldousfill} proves the uniform minimum-degree
bound and its Aldous--Fill corollary.
Except in \Cref{sec:aldousfill}, the degree parameter $d\ge3$ is
fixed. In that section, all finite-order estimates are used in their
exact form when $d$ varies with $n$.


\section{Fiedler vectors and paths}\label{sec:prelim}

The order of a graph $G$ is its number of vertices. We write
$\deg_G(v)$ for the degree of a vertex $v$, $\delta(G)$ for the
minimum degree, and $\dm(G)$ for the diameter, that is, the largest
distance between two vertices.

Let $A$ be the adjacency matrix of $G$ and $\Delta$ the diagonal
matrix of its vertex degrees. The
\emph{Laplacian matrix} is $L(G)=\Delta-A$. Its second smallest
eigenvalue is the \emph{algebraic connectivity}, denoted by
$\mu=\mu(G)$. An eigenvector corresponding to $\mu(G)$ is called a
\emph{Fiedler vector}. If $G$ is $d$-regular, then $L(G)=dI-A$ and
its transition matrix is $A/d$. In this case, $\mu(G)$ is also the
\emph{adjacency spectral gap}, the difference between the two
largest eigenvalues of $A$, whereas the spectral gap of the
transition matrix is $\mu(G)/d$.

A \emph{cut vertex} is a vertex whose deletion disconnects the graph,
and a \emph{bridge} is an edge whose deletion disconnects it.
A \emph{block} is a maximal connected subgraph with no cut vertex;
in particular, each bridge forms a block on two vertices. The
\emph{block-tree} is the bipartite incidence tree whose nodes are
the blocks and the cut vertices, with a block adjacent to each cut
vertex it contains. If $G$ has at least two blocks and its block-tree
is a path, we call $G$ \emph{path-like}. Its two pendant blocks are
called \emph{end blocks}. A connected subgraph consisting of several
blocks near one end is called an \emph{end subgraph}; it need not
be a single block.

Recall that $L_d=K_{d+1}-e$. We call the endpoints of its missing
edge the left and right \emph{singleton vertices}; the remaining
$d-1$ vertices induce $K_{d-1}$. The names ``left'' and ``right''
specify an orientation of the block, not an additional
graph-theoretic condition.

This section collects a few simple facts about Fiedler vectors and paths that will be used repeatedly.

We compare a given vector with the Fiedler vector of a path and show that if their corresponding quadratic forms are close, then the vector itself must be close to the Fiedler vector of the path. We also obtain basic bounds on the size and signs of the components of a Fiedler vector of a graph. These facts will later help us turn a small algebraic connectivity into information about the order and structure of the vertices of the graph.

For a real vector $\x=(x_1,\ldots,x_n)^\top$, the quantity
$\x^\top L(G)\x/\|\x\|^2$ is a {\em Rayleigh quotient}. It is well
known that
\begin{align}
 \mu(G)&=\min_{\x\ne\0,\,\x\perp\1}
       \frac{\x^\top L(G)\x}{\|\x\|^2},\label{eq:Rayleigh}\\
 \x^\top L(G)\x&=\sum_{ij\in E(G)}(x_i-x_j)^2.\label{eq:quadratic}
\end{align}
At a vertex $i$ of degree $d_i$, a Fiedler vector satisfies the
{\em eigen-equation}
\begin{equation}\label{eq:eigenequation}
 \mu x_i=d_ix_i-\sum_{j:\,ij\in E(G)}x_j
        =\sum_{j:\,ij\in E(G)}(x_i-x_j).
\end{equation}

We will also use the sign of the sums over initial segments of an
ordered vector. If $\x$ is nonzero, nondecreasing and has mean zero,
put $f_i=-\sum_{j=1}^i x_j$ for $1\le i<n$. Then
\begin{equation}\label{eq:AFprefix}
 f_i=\frac1n\sum_{j=1}^i\sum_{k=i+1}^n(x_k-x_j)>0
 \qquad(1\le i<n).
\end{equation}
All summands are nonnegative, and equality would force every component
to have the same value.

For $\bar y=n^{-1}\sum_i y_i$, centering gives
\begin{equation}\label{eq:centering}
	\|\y-\bar y\1\|^2
	=\sum_i y_i^2-\frac1n\left(\sum_i y_i\right)^2.
\end{equation}
Since $L(G)\1=\0$, centering does not change the quadratic form.
Therefore, for every nonconstant real vector $\y$,
\begin{equation}\label{eq:centeredRayleigh}
	\mu(G)\le
	\frac{\y^\top L(G)\y}{\|\y-\bar y\1\|^2}.
\end{equation}
Indeed, $\y-\bar y\1$ is nonzero and orthogonal to $\1$, so
\eqref{eq:centeredRayleigh} follows directly from
\eqref{eq:Rayleigh}.
These facts are recalled in \cite[Section~2]{AbGhDiam}. The following
consequence will be used when changing edges. Let $G'$ have the same order
as $G$, let $\x$ be a mean-zero unit Fiedler vector of $G$, and put
$\mu=\mu(G)$. A nonconstant vector $\y$ proves $\mu(G')<\mu$ if
\begin{equation}\label{eq:strictcomparison}
 \y^\top L(G')\y-\mu\|\y\|^2
          +\frac\mu n\left(\sum_i y_i\right)^2<0.
\end{equation}
Indeed, the left side is
$\y^\top L(G')\y-\mu\|\y-\bar y\1\|^2$.
Also, if an edge change leaves the Rayleigh quotient of $\x$ equal
to $\mu$, but $\x$ no longer satisfies the eigen-equations, then
$\mu(G')<\mu$. Equality in the minimum Rayleigh quotient would
otherwise make $\x$ a Fiedler vector of~$G'$.

For $P_n$, the path graph on $n$ vertices, we know that (see
\cite{fiedler1973algebraic})
\[
p_n:=\mu(P_n)=2(1-\cos(\pi/n)),
\]
and by \cite[Theorem~5.6.1]{Spielman}, its increasing unit Fiedler vector has
components
\begin{equation}\label{eq:pathvector}
	v_i=-\sqrt{\frac2n}\cos\frac{(2i-1)\pi}{2n},
	\qquad i=1,\ldots,n.
\end{equation}
Also define
\[
h_i=v_{i+1}-v_i,\qquad i=1,\ldots,n-1.
\]
From \eqref{eq:Rayleigh} and \eqref{eq:centering} it then follows that
\begin{equation}\label{eq:pathineq}
	\y^\top L(P_n)\y=\sum_{i=1}^{n-1}(y_{i+1}-y_i)^2
	\ge p_n\|\y-\bar y\1\|^2.
\end{equation}
We need the following lemma.

\begin{lemma}\label{lem:pathstable}
	Let $n\ge3$ and $\theta\ge0$. Suppose there is a constant $C>0$ such that
	\[
	\left|\|\y-\bar y\1\|^2-1\right|\le C\theta,
	\qquad
	0\le\y^\top L(P_n)\y-p_n\|\y-\bar y\1\|^2
	\le C\theta/n^2.
	\]
	Then there is a constant $C'>0$, depending only on $C$, such that,
	for sufficiently small $\theta$, after changing the sign of $\y$ if
	necessary,
	\begin{equation}\label{eq:pathclose}
		\sum_{i=1}^{n-1}(y_{i+1}-y_i-h_i)^2
		\le C'\theta/n^2.
	\end{equation}
	No sign change is needed when $\y$ is nondecreasing. If
	$A\subseteq\{1,\ldots,n-1\}$ satisfies
	\begin{equation}\label{eq:smallset}
		\sum_{i\in A}(y_{i+1}-y_i)^2\le C\theta/n^2,
	\end{equation}
	then, after increasing $C'$ if necessary,
	\[
	|A|\le C'\theta^{1/3}n.
	\]
\end{lemma}

\begin{proof}
Put $L=L(P_n)$. The next positive eigenvalue after $p_n$ is
$\widetilde p_n=2(1-\cos(2\pi/n))$, so
\begin{equation}\label{eq:pathgap1}
 \frac{\widetilde p_n}{p_n}
 =4\cos^2\frac{\pi}{2n}\ge3\qquad(n\ge3).
\end{equation}
We first prove the following estimate for a mean-zero unit vector $\z$,
after choosing its sign so that $\langle\z,\bv\rangle\ge0$:
\begin{equation}\label{eq:pathunitstable}
 (\z-\bv)^\top L(\z-\bv)
 \le2\bigl(\z^\top L\z-p_n\bigr).
\end{equation}
Write $\z=c\bv+\bu$, where $\bu\perp\1,\bv$ and $c\ge0$.
Then $c^2+\|\bu\|^2=1$. Every eigenvalue contributing to $\bu$
is at least $3p_n$, and $\lambda+p_n\le2(\lambda-p_n)$ for
$\lambda\ge3p_n$. Hence
\[
 \bu^\top L\bu+p_n\|\bu\|^2
 \le2\bigl(\bu^\top L\bu-p_n\|\bu\|^2\bigr).
\]
Since $(1-c)^2\le1-c^2=\|\bu\|^2$, we obtain
\[
 (\z-\bv)^\top L(\z-\bv)
 =p_n(1-c)^2+\bu^\top L\bu
 \le2\bigl(\z^\top L\z-p_n\bigr),
\]
which proves \eqref{eq:pathunitstable}.

Now put $s=\|\y-\bar y\1\|$. We may assume $\theta\le1$ and
$C\theta\le1/2$, so $s>0$. Change the sign of $\y$ if necessary
so that $\langle\y,\bv\rangle\ge0$, and apply
\eqref{eq:pathunitstable} to $\z=(\y-\bar y\1)/s$.
The inequality $(u+v)^2\le2u^2+2v^2$, applied to each path difference,
and $\y-\bar y\1-\bv=s(\z-\bv)+(s-1)\bv$ give
\begin{align*}
 (\y-\bar y\1-\bv)^\top L(\y-\bar y\1-\bv)
 &\le2s^2(\z-\bv)^\top L(\z-\bv)+2p_n(s-1)^2\\
 &\le4\bigl(\y^\top L\y-p_ns^2\bigr)+2p_n(s-1)^2\\
 &\le\frac{(4C+2\pi^2C^2)\theta}{n^2}.
\end{align*}
For the last step we used
$|s-1|=|s^2-1|/(s+1)\le C\theta$ and $p_n\le\pi^2/n^2$.
Adding a constant does not change path differences, so this proves
\eqref{eq:pathclose}. If $\y$ is nondecreasing, then
\[
 \langle\y,\bv\rangle
 =\frac1n\sum_{i<j}(y_j-y_i)(v_j-v_i)>0,
\]
because $\bv$ is strictly increasing and $\y$ is not constant.
Thus no sign change is needed.

It remains to prove the last assertion. From
	\eqref{eq:pathvector},
	\[
	h_i
	=
	2\sqrt{\frac2n}
	\sin\frac{\pi}{2n}
	\sin\frac{\pi i}{n}.
	\]
	Put
	\[
	r_i=\min\{i,n-i\}.
	\]
	Using $\sin t\ge2t/\pi$ on $[0,\pi/2]$, we obtain
	$\sin\frac{\pi}{2n}\ge\frac1n,$ and 	$\sin\frac{\pi i}{n}
	=	\sin\frac{\pi r_i}{n}
	\ge\frac{2r_i}{n}.$
	Consequently
	\begin{equation}\label{eq:hlower}
		h_i\ge
		\frac{4\sqrt2\,r_i}{n^{5/2}}.
	\end{equation}
	
	Let $|A|=k$. Each positive integer occurs at most twice among the
	numbers $r_i$. Hence, if their values on $A$ are arranged increasingly,
	the $j$th one is at least $j/2$. By \eqref{eq:hlower},
	\begin{equation}\label{eq:cubicmass}
		\sum_{i\in A}h_i^2
		\ge \frac{8}{n^5}\sum_{j=1}^k j^2
		\ge \frac{8}{3}\frac{k^3}{n^5}.
	\end{equation}
	On the other hand, \eqref{eq:pathclose} and
	\eqref{eq:smallset} give
	\[
	\sum_{i\in A}h_i^2
	\le
	2\sum_{i\in A}(y_{i+1}-y_i)^2
	+2\sum_{i=1}^{n-1}(y_{i+1}-y_i-h_i)^2
	\le
	\frac{2(C+C')\theta}{n^2}.
	\]
	Together with \eqref{eq:cubicmass}, this gives
	\[
	k^3\le\frac34(C+C')\theta n^3.
	\]
	Thus
	\[
	|A|
	\le
	\left(\frac34(C+C')\right)^{1/3}
	\theta^{1/3}n.
	\]
	Increasing $C'$ once more proves the assertion.
\end{proof}

\begin{lemma}\label{lem:smallentries}
	Let $\x$ be a unit Fiedler vector of a connected graph $G$ of order $n$,
	and write $\mu=\mu(G)$. Then
	\begin{equation}\label{eq:edgeestimate}
		\max_i |x_i|\le \sqrt{(n-1)\mu}.
	\end{equation}
	Consequently, if
	$\mu\le \frac{C}{n^2},$
	then
	$\max_i|x_i|\le \sqrt{\frac{C}{n}},$
	and at least $n/(2C)$ components of $\x$ are positive and at least
	$n/(2C)$ components are negative.
\end{lemma}

\begin{proof}
	The first estimate is a quantitative form of the path argument used in
	\cite[Lemma~2.2]{AbGhDiam}. Since $\x\perp\1$, for every $i$ with
	$x_i>0$ there is a vertex $j$ with $x_j\le0$. Let
	\[
	v_i=v_0,v_1,\ldots,v_\ell=v_j
	\]
	be a shortest path from $v_i$ to $v_j$. Then $\ell\le n-1$, and by
	Cauchy--Schwarz,
	\[
	x_i^2\le (x_i-x_j)^2
	=\left(\sum_{r=0}^{\ell-1}
	(x_{v_r}-x_{v_{r+1}})\right)^2
	\le \ell\sum_{r=0}^{\ell-1}
	(x_{v_r}-x_{v_{r+1}})^2.
	\]
	Since $\x$ is a unit Fiedler vector,
	\[
	\sum_{uv\in E(G)}(x_u-x_v)^2=\mu.
	\]
	Therefore
	$x_i^2\le (n-1)\mu.$
	The same argument, with the signs reversed, applies when $x_i<0$.
	This proves \eqref{eq:edgeestimate}. 
	
	Now suppose that $\mu\le C/n^2$, and put
	$M=\max_i|x_i|.$
	By \eqref{eq:edgeestimate},
	$M^2\le (n-1)\frac{C}{n^2}\le\frac{C}{n}.$
	Since $\|\x\|=1$,
	\[
	1=\sum_i x_i^2\le M\sum_i|x_i|.
	\]
	Also $\sum_i x_i=0$, and hence
	\[
	\sum_{x_i>0}x_i
	=-\sum_{x_i<0}x_i
	=\frac12\sum_i|x_i|
	\ge\frac1{2M}.
	\]
	Every positive component is at most $M$, so the number of positive
	components is at least
	$\frac{1}{2M^2}\ge\frac{n}{2C}.$
	The same argument gives at least $n/(2C)$ negative components.
\end{proof}

\begin{lemma}\label{lem:positive}
	If $Q$ is a real symmetric positive definite matrix whose off-diagonal
	entries are nonpositive, then $Q^{-1}\1>\0$ componentwise.
\end{lemma}

\begin{proof}
	Consider the function
	\[
	f(\y)=\frac12\y^\top Q\y-\1^\top\y,
	\qquad \y\in\mathbb R^m.
	\]
	Since $Q$ is positive definite, $f$ is strictly convex and therefore
	has a unique minimizer. Its gradient is $Q\y-\1$, so the minimizer is
	$\y_0=Q^{-1}\1.$
	
	We first show that $\y_0\ge\0$. For any $\y\in\mathbb R^m$, let
	$|\y|$ denote the vector obtained by replacing every component of
	$\y$ by its absolute value. Since $q_{ij}\le0$ for $i\ne j$,
	replacing $\y$ by $|\y|$ cannot increase the quadratic term
	$\y^\top Q\y$. Also,
	$-\1^\top|\y|\le-\1^\top\y$, with strict inequality if some component
	of $\y$ is negative. Hence $f(|\y|)\le f(\y)$, with strict inequality
	if $\y$ has a negative component. Since $\y_0$ is the unique
	minimizer, it follows that $\y_0\ge\0$.
	
	It remains to show that every component is positive. Suppose
	$(\y_0)_i=0$ for some $i$. Then, since $(\y_0)_j\ge0$ and
	$q_{ij}\le0$ for $j\ne i$,
	\[
	(Q\y_0)_i
	=q_{ii}(\y_0)_i+\sum_{j\ne i}q_{ij}(\y_0)_j
	=\sum_{j\ne i}q_{ij}(\y_0)_j
	\le0.
	\]
	But $Q\y_0=\1$, so $(Q\y_0)_i=1$, a contradiction. Therefore
	$\y_0>\0$, and hence $Q^{-1}\1>\0$.
\end{proof}

\section{Local estimates and comparison graphs}\label{sec:numerical}
In this section, we compare ordered vectors with vectors on a path
and construct graphs with small algebraic connectivity. These facts
will be used to determine the structure of exact minimizers.

Throughout this section, put $D=d+1$ and $\al=d-1$. We compare a
vector on $G$ with a vector on $P_n$, obtained by replacing certain
consecutive components by equally spaced values. We keep track of the
difference between the two quadratic forms, since it is needed for
\Cref{thm:stability}.

\subsection{Edges crossing consecutive cuts}
We put the vertices in order and count the edges crossing each cut.
These counts will help us estimate algebraic connectivity.

Let $G$ be a graph with $V(G)=\{1,\ldots,n\}$.
Let $\x$ be any real vector on $V(G)$. Label the vertices so that
$x_1\le\cdots\le x_n$, and put
\[
 a_i=x_{i+1}-x_i\ge0,\qquad 1\le i\le n-1.
\]
Let $q_i$ be the number of edges joining $\{1,\ldots,i\}$ to $\{i+1,\ldots,n\}$. We call this the cut at position $i$, and set $q_0=q_n=0$.
Let $K_{ij}$ be the number of edges crossing both the cut at $i$ and the
cut at $j$. 
Equivalently, $K_{ij}$ counts the edges $uv$, $u<v$, for which
\[
u\le\min\{i,j\},\qquad v>\max\{i,j\}.
\]

In particular, $K_{ij}=K_{ji}$ and $K_{ii}=q_i$.
For an edge $uv$ with $u<v$, we have $x_v-x_u=\sum_{i=u}^{v-1}a_i$.
Hence
\[
(x_v-x_u)^2
=\sum_{i=u}^{v-1}\sum_{j=u}^{v-1}a_i a_j.
\]
Now sum this identity over all edges $uv\in E(G)$. For fixed $i,j$,
the term $a_i a_j$ occurs once for every edge that crosses both cuts
$i$ and $j$, and there are exactly $K_{ij}$ such edges. Therefore
\begin{equation}\label{eq:full}
	\x^\top L(G)\x
	=\sum_{uv\in E(G)}(x_u-x_v)^2
	=\sum_{i,j=1}^{n-1}K_{ij}a_i a_j.
\end{equation}
Since $K_{ij}\ge0$ and $a_i\ge0$, every term in the last sum is
nonnegative. We may therefore keep only selected terms whenever we
want a lower bound.

\begin{lemma}\label{lem:counts}
If $\delta(G)\ge d$, then
\begin{equation}\label{eq:cutcount}
 q_a+q_{a+t}\ge t(D-t)
 \qquad(0\le a<a+t\le n,\ 1\le t\le d).
\end{equation}
If $q_a=k$, then, for $1\le s\le t\le d$ and $a+t\le n-1$,
\begin{equation}\label{eq:overlap}
 K_{a+s,a+t}\ge s(D-t)-k.
\end{equation}
For every  graph, including the end cuts $1\le i\le n-1$,
with a term outside the range interpreted as zero,
\begin{equation}\label{eq:neighborcount}
 K_{i-1,i}+K_{i,i+1}\ge q_i-1.
\end{equation}
\end{lemma}
\begin{proof}
Consider the $t$ vertices in positions $a+1,\ldots,a+t$. Their degrees
have sum at least $dt$. Edges between these vertices account for at most
$t(t-1)$ in that sum. At least $dt-t(t-1)=t(D-t)$ edges therefore leave
the set. Every one of them crosses the cut at $a$ or the cut at $a+t$,
which proves \eqref{eq:cutcount}.

For \eqref{eq:overlap}, consider only the first $s$ of these vertices.
Their degrees have sum at least $ds$. Edges inside the set account for
at most $s(s-1)$, edges to the next $t-s$ vertices account for at most
$s(t-s)$, and edges to positions at most $a$ account for at most $k$.
There are consequently at least
\[
 ds-s(s-1)-s(t-s)-k=s(D-t)-k
\]
edges from these $s$ vertices to positions greater than $a+t$. Each of
these edges crosses both indicated cuts.

Finally, an edge crossing the cut at $i$ must also cross one of its two
neighboring cuts, unless the edge is $i(i+1)$. Since the graph is simple,
there is at most one exception. This proves \eqref{eq:neighborcount}.
\end{proof}

Call a cut {\em small} if $q_i\le d-2$. By \eqref{eq:cutcount}, two
small cuts cannot be at distance $2,\ldots,d-1$, because their sizes
would have sum at most $2(d-2)$, whereas
$t(D-t)\ge2(d-1)$ in that range. Thus small cuts occur singly or in
adjacent pairs. The first and last $d$ cuts are not small: setting
$a=0$ in \eqref{eq:cutcount} gives
$q_i\ge i(D-i)\ge d$ for $1\le i\le d$, and the other end is the same.

At the first cut $a$ of each group of small cuts, select the interval
of difference indices
\begin{equation}\label{eq:windows}
 B_a=\{a,\ldots,a+r-1\},\qquad
 r=\begin{cases}D,&q_a=1,\\ d,&2\le q_a\le d-2.\end{cases}
\end{equation}
These intervals are disjoint and cover all small cuts. Indeed, group
beginnings are at least $d$ apart. If $q_a=1$, \eqref{eq:cutcount}
also excludes a small cut at distance $d$ and an adjacent small cut:
the two sizes would sum to at most $d-1$. The endpoint cut bounds
ensure that every interval lies in $\{1,\ldots,n-1\}$.

\subsection{Inequalities on a selected interval}
Now, we estimate the contribution from a short interval and show why intervals
starting at a bridge are special.

For an interval $B=B_a$ of length $r$, write
\begin{align*}
	 w_s&=a_{a+s}\quad(0\le s<r),\\
	\Delta_B&=x_{a+r}-x_a=\sum_{s=0}^{r-1}w_s,\\
	 Q_B&=\sum_{i,j\in B}K_{ij}a_i a_j.
\end{align*}
Here $Q_B$ retains the terms of \eqref{eq:full} with both indices in $B$;
it is not the quadratic form of the induced subgraph.

For positive numbers $c_1,\ldots,c_m$ and real numbers
$g_1,\ldots,g_m$, put $R=\sum_jc_j^{-1}$ and $\Delta=\sum_jg_j$.
Expanding the squares gives the exact identity
\begin{equation}\label{eq:weighted-squares}
 \sum_{j=1}^m c_jg_j^2-\frac{\Delta^2}{R}
 =\sum_{j=1}^m c_j\left(g_j-\frac{\Delta}{Rc_j}\right)^2.
\end{equation}
Thus $\sum_jc_jg_j^2\ge\Delta^2/R$, with equality precisely when
$g_j=\Delta/(Rc_j)$ for every $j$. This is the weighted
Cauchy--Schwarz inequality, with its nonnegative remainder retained.

\begin{lemma}\label{lem:local}
Every interval in \eqref{eq:windows} satisfies
\begin{equation}\label{eq:localbase}
 Q_B\ge\al\frac{\Delta_B^2}{|B|}.
\end{equation}
There is a constant $\kappa_d>0$ such that, if $q_a\ge2$, then
\begin{equation}\label{eq:localstrict}
 Q_B\ge(\al+\kappa_d)\frac{\Delta_B^2}{d}.
\end{equation}
If $q_a=1$, put $t_B=\Delta_B/D$ and
$R_B=Q_B-\al D t_B^2$. Then
\begin{equation}\label{eq:bridgeends}
 (a_a-\al t_B)^2+(a_{a+d}-t_B)^2\le2R_B.
\end{equation}
\end{lemma}
\begin{proof}
	We consider separately the intervals beginning at a bridge and the other
	intervals.
	
	\emph{Case 1: $q_a=1$.}
	Here $r=D$. Put $Z=w_1+\cdots+w_d$, and define $D$ numbers
	$v_1,\ldots,v_D$ by $v_1=0$ and
	$v_{j+1}-v_j=w_j$ for $1\le j\le d$. Thus $v_D=Z$.
	The number of edges of $K_D$ crossing both cuts $s\le t$ is
	$s(D-t)$. Hence \eqref{eq:overlap}, summed against $w_sw_t$, gives
	\begin{equation}\label{eq:cliquecomparison}
		Q_B\ge w_0^2+\bv^\top L(K_D)\bv-Z^2.
	\end{equation}
	Here the subtracted $1$ contributes
	$\sum_{s,t=1}^d w_sw_t=Z^2$, while all terms involving $w_0$ except
	$w_0^2$ have been omitted and are nonnegative.
	
	For $2\le j\le D-1$, put $u_j=v_j-Z/2$. Using
	$\bv^\top L(K_D)\bv
	=D\sum_jv_j^2-(\sum_jv_j)^2$, we obtain
	\begin{align}
		\bv^\top L(K_D)\bv
		&=\frac D2Z^2+D\sum_{j=2}^{D-1}u_j^2
		-\left(\sum_{j=2}^{D-1}u_j\right)^2 \notag\\
		&\ge \frac D2Z^2+2\sum_{j=2}^{D-1}u_j^2.
		\label{eq:cliquesquares}
	\end{align}
	The last inequality follows from Cauchy--Schwarz. In particular,
	\eqref{eq:cliquecomparison} gives
	$Q_B\ge w_0^2+\al Z^2/2$, and weighted Cauchy--Schwarz yields
	\[	Q_B\ge \frac{\al}{D}(w_0+Z)^2.	\]
	This proves \eqref{eq:localbase}.
	
	For the stronger estimate, retain the squares in
	\eqref{eq:cliquesquares} and put $e_0=w_0-\al t_B$.
	Apply \eqref{eq:weighted-squares} with weights $1,\al/2$
	and differences $w_0,Z$. Since $w_0+Z=Dt_B$, this gives
	\[
	R_B\ge\frac D2e_0^2+2\sum_{j=2}^{D-1}u_j^2,
	\qquad
	w_d-t_B=-u_{D-1}-\frac{e_0}{2}.
	\]
	Therefore
	$e_0^2+(w_d-t_B)^2
	\le \frac32e_0^2+2u_{D-1}^2\le2R_B$,
	which proves \eqref{eq:bridgeends}.
	
	\emph{Case 2: $q_a=k\ge2$ and $d\ge5$.}
	Here $2\le k\le d-2$ and $r=d$. Put
	$Z=w_2+\cdots+w_{d-1}$. Since $K_{ii}=q_i$, the first two diagonal
	terms in $Q_B$ are
	$K_{aa}w_0^2+K_{a+1,a+1}w_1^2
	=q_aw_0^2+q_{a+1}w_1^2$.
	As $q_a=k$ and \eqref{eq:cutcount} gives $q_{a+1}\ge d-k$, these terms
	are at least $kw_0^2+(d-k)w_1^2$.
	
	To estimate the other terms, consider $D$ numbers
	$v_1,\ldots,v_D$ with consecutive differences
	$0,w_2,\ldots,w_{d-1},0$ and $v_1=0$. Thus
	$v_1=v_2=0$ and $v_{D-1}=v_D=Z$.
	For $1\le s\le t\le d$, \eqref{eq:overlap} gives
	$K_{a+s,a+t}\ge s(D-t)-k$.
	Summing these inequalities against the corresponding products of the
	consecutive differences gives a contribution of at least
	$\bv^\top L(K_D)\bv-kZ^2$.
	Indeed, the terms $s(D-t)$ give the quadratic form of $K_D$, while the
	subtracted $k$ contributes
	$k(0+w_2+\cdots+w_{d-1}+0)^2=kZ^2$.
	
	Now
	$\bv^\top L(K_D)\bv
	=D\sum_i v_i^2-(\sum_i v_i)^2$.
	Subject to $v_1=v_2=0$ and $v_{D-1}=v_D=Z$, this is minimized when
	$v_3=\cdots=v_{D-2}=Z/2$, and the minimum is $DZ^2$.
	Hence
	\[
	Q_B\ge kw_0^2+(d-k)w_1^2+(D-k)Z^2.
	\]
	Since $\Delta_B=w_0+w_1+Z$, weighted Cauchy--Schwarz gives
	\begin{equation}\label{eq:nonbridge3}
		Q_B\ge kw_0^2+(d-k)w_1^2+(D-k)Z^2
		\ge
		\frac{\Delta_B^2}
		{1/k+1/(d-k)+1/(D-k)}.
	\end{equation}
	
	The reciprocal sum in the denominator is convex for
	$k\in[2,d-2]$, so its maximum occurs at an endpoint. The larger
	endpoint value is $S_d=5/6+1/(d-2)$. Moreover,
	\[
	\frac d{d-1}-S_d
	=\frac16-\frac1{(d-1)(d-2)}>0
	\qquad(d\ge5).
	\]
	Thus $\kappa_d=d/S_d-(d-1)>0$, proving
	\eqref{eq:localstrict}.
	
	\emph{Case 3: $d=4$ and $q_a\ge2$.}
	Now $k=2$. By \eqref{eq:overlap},
	\[
	Q_B\ge
	2w_0^2+2w_1^2+4w_2^2+4w_3^2
	+2w_1w_2+4w_2w_3.
	\]
	If $T=w_1+w_2+w_3$, the last five terms equal
	\[
	\frac43T^2
	+\frac23(w_1-2w_3-w_2/2)^2
	+\frac52w_2^2.
	\]
	Hence
	$Q_B\ge2w_0^2+4T^2/3\ge4(w_0+T)^2/5$.
	Thus \eqref{eq:localstrict} holds with $\kappa_4=1/5$.
	For $d=3$, every small cut has size one, so there is no further case; we may set $\kappa_3=1$.
\end{proof}

\subsection{An adjacent-pair estimate}
Next, we use pairs of neighbouring positions to control the terms left over
from the interval estimates.

The following estimate will replace separate counts of small
differences in three applications. For $x,y,s\ge0$ and $c\ge1$,
\[
 (x-s)^2+(y-cs)^2+2xy
 =(x+y-cs)^2+2(c-1)sx+s^2\ge s^2.
\]
Hence, for $t\ge0$,
\begin{equation}\label{eq:adjacentpair}
 s^2\le(x-s)^2+2(y-ct)^2+2c^2(s-t)^2+2xy.
\end{equation}
Indeed, write $y-cs=(y-ct)+c(t-s)$ and use
$(u+v)^2\le2u^2+2v^2$.

Use disjoint selected intervals, allowing also an initial interval.
Write $U$ for the uncovered indices, $A$ for the union of the
nonbridge selected intervals, and $J$ for the endpoints of the
bridge intervals and of the initial interval, when present.
A neighbour of an index in $U$ belongs to $U$, $A$, or $J$.
Put $V=U\cup A$.
\begin{lemma}\label{lem:adjacent}
Suppose $a_i\ge0$, $1\le i\le M$, are ordered differences and
$q_i\ge2$ on $U$. Choose $p_i\ge0$, and choose $1\le c_j\le\alpha$
for $j\in J$, with $c_i=1$ on $U$. If $\mathcal R\ge0$ satisfies
 $$\sum_{i\in U}(a_i-p_i)^2+\sum_{j\in J}(a_j-c_jp_j)^2
       +\sum_{i\in A}p_i^2 +
 \sum_{\substack{1\le i<M\\\{i,i+1\}\cap U\ne\varnothing}}
       K_{i,i+1}a_i a_{i+1}\le C_d\mathcal R,$$
then
\begin{equation}\label{eq:adjacentcontrol}
 \sum_{i\in V}p_i^2\le C_d\left(\mathcal R+
                   \sum_{i=1}^{M-1}(p_{i+1}-p_i)^2\right).
\end{equation}
\end{lemma}
\begin{proof}
For each $i\in U$, \eqref{eq:neighborcount} and integrality give a
neighbour $j$ with $K_{ij}\ge1$. If $j\in U\cup J$, apply
\eqref{eq:adjacentpair} with $x=a_i$, $y=a_j$, $s=p_i$, $t=p_j$,
and $c=c_j$. If $j\in A$, use instead
\[
 p_i^2\le2(p_i-p_j)^2+2p_j^2.
\]
Sum and add $\sum_{i\in A}p_i^2$. Each index can be chosen as a
neighbour at most twice, and each adjacent pair is used at most
twice. For a chosen pair, $a_i a_j\le K_{ij}a_i a_j$.
Thus the controlled terms have bounded multiplicity, proving
\eqref{eq:adjacentcontrol}. This includes the end indices.
\end{proof}

\subsection{Comparison with a path}
We replace values inside each selected interval by evenly spaced values,
allowing us to compare the graph with a path.

Define a vector $\y$ by replacing the components on each selected interval
by equally spaced values. More precisely, if $B_a$ has length $r$, set
\begin{equation}\label{eq:interpolation}
 y_{a+j}=x_a+\frac j r(x_{a+r}-x_a),\qquad 0\le j\le r,
\end{equation}
and let $y_i=x_i$ at all other positions. The intervals of difference
indices are disjoint, so these definitions agree even when two vertex
intervals have a common endpoint. Let $U$ be the difference indices not
in any selected interval, and put $b_i=y_{i+1}-y_i$.


\begin{lemma}\label{lem:pathcomparison}
	Let $\x$ be any ordered real vector, and let $\y$ be the interpolated
	vector defined above. Then
	\begin{equation}\label{eq:pathcomparison}
		\x^\top L(G)\x
		\ge \al\sum_{i=1}^{n-1}b_i^2
		=\al\,\y^\top L(P_n)\y.
	\end{equation}
	Moreover,
	\begin{equation}\label{eq:vectorerror}
		\|\x-\y\|^2
		\le \frac{dD}{\al}\,\x^\top L(G)\x.
	\end{equation}
	These inequalities hold for every ordered real vector $\x$, not only
	for a Fiedler vector.
\end{lemma}

\begin{proof}
	On each interval $B$, all differences $b_i$ are equal to
	$\Delta_B/|B|$, and hence
	\[
	\sum_{i\in B}b_i^2=\frac{\Delta_B^2}{|B|}.
	\]
	On $U$, we have $b_i=a_i$ and $q_i\ge\al$. Therefore
	\Cref{lem:local} together with \eqref{eq:full} gives
	\eqref{eq:pathcomparison}.
	
	For the second inequality, consider an interval
	$B_a=\{a,\ldots,a+r-1\}$. Both $x_{a+j}$ and $y_{a+j}$ lie between
	$x_a$ and $x_{a+r}$, so
	$|x_{a+j}-y_{a+j}|\le\Delta_B$. Hence
	\[
	\sum_{j=1}^{r-1}(x_{a+j}-y_{a+j})^2
	\le (r-1)\Delta_B^2.
	\]
	Since $r\le D$, \Cref{lem:local} gives
	\[
	(r-1)\Delta_B^2
	\le \frac{r(r-1)}{\al}Q_B
	\le \frac{dD}{\al}Q_B.
	\]
	Summing over all intervals and using the fact that their contributions
	to \eqref{eq:full} are disjoint gives
	\[
	\|\x-\y\|^2
	\le \frac{dD}{\al}\sum_B Q_B
	\le \frac{dD}{\al}\x^\top L(G)\x,
	\]
	which is \eqref{eq:vectorerror}.
\end{proof}

\subsection{Comparison graphs and an upper bound}\label{sec:upper}
We construct graphs with small algebraic connectivity. Exact
minimality will then give the same upper bound for a minimizing
graph, before its structure is known.

In each class the comparison graphs consist of $m$ copies of $L_d$
in a chain, with two
connected end subgraphs of bounded order. Each end subgraph is joined
to the chain by one bridge, so $n=Dm+O_d(1)$.

In the minimum-degree classes, write $n=D(m+2)+r$, where
$0\le r<D$. Take $K_D$ and $K_{D+r}$ as the two end subgraphs,
with any vertex in each as its attachment vertex. Join these to the
outer singleton vertices of the $L_d$ chain. For all sufficiently
large $n$, this graph has minimum degree exactly $d$.

For odd $d$, we need end subgraphs that give a regular graph at every
sufficiently large even order. We give the two bounded building blocks
explicitly; they will also be used in \Cref{or:lem:comparisons}.
For even $N$ with $D\le N\le2d$, start with the graph on
$\mathbb Z_N$ in which each vertex $i$ is adjacent to
\[
 i\pm1,\ldots,i\pm\frac{d-1}{2},\qquad i+\frac N2
 \quad\pmod N.
\]
These are $d$ distinct neighbours. Delete the edge $01$ and call the
resulting graph $Q_N$. It is connected, since the cycle of edges of
step one becomes a spanning path after this deletion. Its vertices
$0,1$ have degree $d-1$, and all other vertices have degree $d$.

Let $E_d$ be obtained from $K_{d+2}$ by choosing distinct vertices
$r,u,v$, deleting $ru$ and $rv$, and deleting a perfect matching on
the remaining $d-1$ vertices. Such a matching exists because $d-1$
is even. The vertex $r$ has degree $d-1$, all other vertices have
degree $d$, and $E_d$ is connected: every pair of nonadjacent vertices
has a common neighbour.

The possible values $N=D,D+2,\ldots,2d$ represent every even residue
class modulo $D$. Thus, for every sufficiently large even $n$, we can
write
\[
 n=mD+N+2(d+2),\qquad m\ge1.
\]
Join, in order, one copy of $E_d$, $m$ copies of $L_d$, one copy of
$Q_N$, and another copy of $E_d$ by bridges. Use each degree-$(d-1)$
vertex once as a bridge endpoint. The result is connected, simple and
$d$-regular. Its two bounded end subgraphs are $E_d$ and the joined
pair $Q_N,E_d$.

In either class, denote the resulting comparison graph by $G_{n,d}$.
These are the same type of chain comparison used in \cite{AbGhDiam}.
The following calculation gives the required error term directly.

\begin{lemma}\label{lem:comparisonupper}
For every fixed $d\ge3$, there are constants $A_d,n_d>0$ such that,
for every admissible $n\ge n_d$, the comparison graph in either
minimum-degree class, and in the regular class when $d$ is odd,
satisfies
\begin{equation}\label{eq:comparisonerror}
 \mu(G_{n,d})\le\frac{(d-1)\pi^2}{n^2}+\frac{A_d}{n^3}.
\end{equation}
\end{lemma}
\begin{proof}
Let $m\ge2$ be the number of middle copies of $L_d$, and let the
end subgraphs have orders $b_-,b_+$. Thus
$n=Dm+b_-+b_+$, where $b_-+b_+=O_d(1)$.
Put
\[
 x_j=\sqrt{\frac2m}\cos\frac{(2j-1)\pi}{2m},
 \qquad 1\le j\le m,
 \qquad \Delta_j=x_{j+1}-x_j\quad(1\le j<m).
\]
The path formulas in \Cref{sec:prelim} give
\[
 \sum_{j=1}^m x_j=0,\qquad
 \sum_{j=1}^m x_j^2=1,\qquad
 \sum_{j=1}^{m-1}\Delta_j^2=p_m.
\]
Assign the value $x_j$ to the left singleton of the $j$th copy.
For $j<m$, assign $x_j+\Delta_j/D$ to each of its $d-1$ middle
vertices and $x_j+2\Delta_j/D$ to its right singleton. Give every
vertex of the last copy the value $x_m$. Extend the vector constantly
over the left and right end subgraphs, with values $x_1$ and $x_m$,
respectively. Call the resulting vector $\y$.

In the $j$th copy and its outgoing bridge, the two internal drops are
$\Delta_j/D$, each repeated $\alpha=d-1$ times, and the bridge drop
is $\alpha\Delta_j/D$. Thus
\[
 \y^\top L(G_{n,d})\y
 =\frac{\alpha}{D}\sum_{j=1}^{m-1}\Delta_j^2
 =\frac{\alpha}{D}p_m.
\]
All other edges have zero drop. The squared norm on the middle chain
is
\begin{align*}
 \sum_{v\text{ in the chain}}y_v^2
 &=D+2\sum_{j=1}^{m-1}x_j\Delta_j
       +\frac{D+2}{D^2}\sum_{j=1}^{m-1}\Delta_j^2\\
 &=D-\frac{D^2-D-2}{D^2}p_m,
\end{align*}
since $x_m^2=x_1^2$ and
$2\sum_jx_j\Delta_j=x_m^2-x_1^2-\sum_j\Delta_j^2=-p_m$.
The end subgraphs add nonnegative squared mass. The sum of the
components on the chain is $x_m-x_1$, so
\[
 \1^\top\y=(b_--1)x_1+(b_++1)x_m=O_d(m^{-1/2}).
\]
Centering therefore gives
$\|\y-\bar y\1\|^2\ge D-O_d(m^{-2})$.
By \eqref{eq:centeredRayleigh},
\[
 \mu(G_{n,d})
 \le\frac{\alpha p_m/D}{D-O_d(m^{-2})}
 \le\frac{\alpha\pi^2}{D^2m^2}+O_d(m^{-4}).
\]
Finally, $n=Dm+O_d(1)$ gives
$(D^2m^2)^{-1}=n^{-2}+O_d(n^{-3})$, proving the claim.
\end{proof}

\section{Structure of graphs with small algebraic connectivity}\label{sec:stability}
In this section, we show that a graph whose algebraic connectivity
is at most the comparison value, up to a small error, is mostly
made of $L_d$ blocks.

The following theorem makes this precise. The copies of $L_d$
in its conclusion are blocks of the graph.
\begin{theorem}\label{thm:stability}
	For every fixed $d\ge3$, there are constants $C_d,n_d$ such that the
	following holds. Suppose $|V(G)|=n\ge n_d$, $\delta(G)\ge d$, and
	\[
	0\le\e\le1,\qquad n^2\mu(G)\le(d-1)\pi^2+\e.
	\]
	Pairwise vertex-disjoint blocks isomorphic to $L_d$ cover at least
	\begin{equation}\label{eq:coverage}
		n-C_d(\e+n^{-1})^{1/3}n
	\end{equation}
	vertices. Exactly two edges leave each selected block. They are bridges
	and meet the endpoints of its missing edge, one at each endpoint.
	The blocks lie, in order, along one sequence of bridges. Deleting all
	other vertices leaves components that are chains of the selected blocks.
	Moreover,
	\begin{equation}\label{eq:diamstable}
		\dm(G)\ge\frac{3n}{d+1}-C_d(\e+n^{-1})^{1/3}n.
	\end{equation}
\end{theorem}
The main idea is to use the terms of \eqref{eq:full} that were discarded
in the local path comparison. If many positions do not lie in
the desired bridge intervals, these extra nonnegative terms become too
large to be compatible with the assumed upper bound on $\mu$.

Let $\x$ be a nondecreasing unit Fiedler vector, and use the selected
intervals and the interpolated vector $\y$ from
\Cref{sec:numerical}. Thus
$a_i=x_{i+1}-x_i$, $b_i=y_{i+1}-y_i$, and $U$ is the set of difference
indices that do not belong to a selected interval.

For every selected interval $B$, put
\[
R_B=Q_B-\al\frac{\Delta_B^2}{|B|},
\qquad
R=\mu-\al\y^\top L(P_n)\y.
\]
Also, let $\mathcal P$ be the set of pairs $(i,j)$ with $i<j$ such that
$i$ and $j$ do not belong to the same selected interval.

We now separate the terms in \eqref{eq:full}. The terms whose two
indices lie in the same selected interval $B$ have total contribution
$Q_B$. If $i\in U$, the diagonal term is $q_i a_i^2$. All remaining
terms have two distinct indices and are indexed by $\mathcal P$.
Therefore
\[
\mu
=\sum_B Q_B+\sum_{i\in U}q_i a_i^2
+2\sum_{(i,j)\in\mathcal P}K_{ij}a_i a_j.
\]
On the other hand, on a selected interval $B$ all the differences
$b_i$ are equal to $\Delta_B/|B|$, while $b_i=a_i$ for $i\in U$.
Hence
\[
\al\y^\top L(P_n)\y
=\sum_B\al\frac{\Delta_B^2}{|B|}
+\sum_{i\in U}\al a_i^2.
\]
Subtracting these two identities gives
\begin{equation}\label{eq:excess}
	R
	=\sum_B R_B
	+\sum_{i\in U}(q_i-\al)a_i^2
	+2\sum_{(i,j)\in\mathcal P}K_{ij}a_i a_j.
\end{equation}

Every term on the right is nonnegative. Indeed,
\Cref{lem:local} gives $R_B\ge0$, while $q_i\ge\al$ for
$i\in U$. In the last sum, $K_{ij}\ge0$ and $a_i,a_j\ge0$.
The factor $2$ appears because \eqref{eq:full} contains both
$K_{ij}a_i a_j$ and $K_{ji}a_j a_i$ when $i\ne j$.
Thus \eqref{eq:excess} is an exact decomposition of the amount by
which $\mu$ exceeds the path lower bound.

We now show that the difference $R=\mu-\al\,\y^\top L(P_n)\y$ is very small. This also implies that the
interpolated vector $\y$ is very close to the Fiedler vector of the path,
so its consecutive differences are close to the path differences $h_i$.

\begin{lemma}\label{lem:smallremaining}
	Put $\theta=\e+n^{-1}$ and $\eta=\theta^{1/3}$. If $\theta$ is
	sufficiently small, depending only on $d$, then
	\begin{align}
		0\le R&\le C_d\frac{\theta}{n^2},
		\label{eq:smallR}\\
		0\le \y^\top L(P_n)\y-p_n\|\y-\bar y\1\|^2
		&\le C_d\frac{\theta}{n^2},
		\label{eq:smallpathdiff}\\
		\sum_{i=1}^{n-1}(b_i-h_i)^2
		&\le C_d\frac{\theta}{n^2}.
		\label{eq:bdifference}
	\end{align}
\end{lemma}

\begin{proof}
	The hypothesis of \Cref{thm:stability} gives
	$\mu=O_d(n^{-2})$. By \eqref{eq:vectorerror},
	$\|\x-\y\|=O_d(n^{-1})$, and hence
	$\|\y-\bar y\1\|^2=1+O_d(n^{-1})$.
	
	Since
	$p_n=2(1-\cos(\pi/n))=\pi^2/n^2+O(n^{-4})$, it follows that
	\[
	\al p_n\|\y-\bar y\1\|^2
	=\frac{\al\pi^2}{n^2}+O_d(n^{-3}).
	\]
	On the other hand, \eqref{eq:pathcomparison} and the assumption in
	\Cref{thm:stability} give
	\[
	\al\y^\top L(P_n)\y
	\le\mu
	\le\frac{\al\pi^2+\e}{n^2}.
	\]
	Combining these inequalities with \eqref{eq:pathineq} gives
	\eqref{eq:smallR} and \eqref{eq:smallpathdiff}.
	
	\Cref{lem:pathstable} now applies to $\y$. Since $\y$ is
	nondecreasing, no change of sign is needed, and therefore
	\eqref{eq:bdifference} follows.
	
	If $\theta$ is not sufficiently small, the estimates can be made
	trivial by increasing $C_d$.
\end{proof}

\subsection{The number of exceptional difference indices}
We show that only a small number of positions can lie outside the selected
intervals that start at bridges.

We now show that only a small number of difference indices fail to lie in selected intervals beginning at bridges. More precisely, we prove that the corresponding differences have very small total squared sum, which then gives the required bound on the number of such indices.

\begin{lemma}\label{lem:exceptionalindices}
	Let $A$ be the union of the nonbridge selected intervals, and put
	$V=U\cup A$. Then
	\begin{equation}\label{eq:Venergy}
		\sum_{i\in V}b_i^2\le C_d\frac{\theta}{n^2}.
	\end{equation}
	Consequently,
	\begin{equation}\label{eq:Vsize}
		|V|\le C_d\eta n.
	\end{equation}
	In particular, all but $O_d(\eta n)$ difference indices lie in selected
	intervals beginning at bridges.
\end{lemma}

\begin{proof}
	Apply \Cref{lem:adjacent} with $p_i=b_i$ and
	$\mathcal R=R$. On $U$ we have $a_i=b_i$ and
	$q_i\ge\alpha\ge2$. At the first and last indices of a bridge interval,
	take $c_j=\alpha$ and $c_j=1$, respectively. The endpoint errors are
	controlled by \eqref{eq:bridgeends}, while \eqref{eq:localstrict} gives
	\[
	\kappa_d\sum_{i\in A}b_i^2\le R.
	\]
	Every adjacent product involving an index in $U$ occurs in the last sum
	of \eqref{eq:excess}. Hence all the hypotheses of
	\Cref{lem:adjacent} are satisfied, and therefore
	\[
	\sum_{i\in V}b_i^2
	\le
	C_d\left(
	R+\sum_{i=1}^{n-2}(b_{i+1}-b_i)^2
	\right).
	\]
	
	It remains to estimate the second term. 
	Recall that
	\[
	h_i=2\sqrt{\frac2n}\sin\frac{\pi}{2n}
	\sin\frac{\pi i}{n},
	\qquad 1\le i\le n-1.
	\]
	For convenience, set $h_0=h_n=0$, which agrees with the same formula
	at $i=0,n$. Since the sequence $h_i$ is a constant multiple of
	$\sin(\pi i/n)$, we have
	$2h_i-h_{i-1}-h_{i+1}=p_n h_i.$
	Therefore
	\[	\sum_{i=0}^{n-1}(h_{i+1}-h_i)^2
	=p_n\sum_{i=1}^{n-1}h_i^2.	\]
	Also, since $h_i=v_{i+1}-v_i$ and $\bv$ is a unit Fiedler vector of
	$P_n$,
	\[	\sum_{i=1}^{n-1}h_i^2
	=\bv^\top L(P_n)\bv=p_n.	\]
	Hence
	\[	\sum_{i=0}^{n-1}(h_{i+1}-h_i)^2=p_n^2=O(n^{-4}).	\]
	Using \eqref{eq:bdifference},
	\[	\sum_{i=1}^{n-2}(b_{i+1}-b_i)^2
	\le	8\sum_{i=1}^{n-1}(b_i-h_i)^2+2p_n^2
	\le	C_d\frac{\theta}{n^2}.	\]
	Together with \eqref{eq:smallR}, this proves
	\eqref{eq:Venergy}.
	
	Finally, \Cref{lem:pathstable} applies to the set $V$, since
	\eqref{eq:Venergy} gives the required bound on
	$\sum_{i\in V}b_i^2$. Hence
	$|V|\le C_d\theta^{1/3}n=C_d\eta n$, proving
	\eqref{eq:Vsize}.
\end{proof}


\subsection{Consecutive bridges and the induced blocks}
We show that most vertices lie between consecutive bridges and form the
required $L_d$ blocks.

We now use the preceding estimate to locate most vertices between
consecutive bridge cuts. We first show that almost all positions occur
between two consecutive bridges at distance $D$, and then show that
almost all such intervals induce copies of $L_d=K_D-e$.

If $H$ is a subgraph of $G$, we call a vertex of $H$ an
\emph{attachment vertex} if it is incident with an edge having its
other endpoint outside $H$. If $H$ lies between two ordered bridges,
the endpoints of these bridges that belong to $H$ are called the left
and right attachment vertices of $H$.

\begin{lemma}\label{lem:bridgeblocks}
	Let $a_1<\cdots<a_k$ be the positions of all cuts of size one. Then,
	apart from at most $C_d\eta n$ vertices, the graph is covered by the
	sets of $D$ vertices lying between two consecutive cuts of size one
	whose positions differ by $D$.
	
	Moreover, if two such cuts occur at positions $a$ and $a+D$, and
	\[
	T=\{a+1,\ldots,a+D\},
	\]
	then $G[T]$ is either $K_D-e=L_d$ or $K_D$. In the first case, the two
	edges leaving $T$ are bridges and are attached to the two endpoints of
	the missing edge.
\end{lemma}

\begin{proof}
	Each cut of size one has a selected interval of length $D$, and these
	selected intervals are disjoint. Hence the number of difference
	indices outside them is
	\[
	r=n-1-Dk.
	\]
	By \Cref{lem:exceptionalindices}, these indices belong to the
	exceptional set $V$, and therefore
	$r\le C_d\eta n$.
	
	Assume first that $k\ge1$. The positions not contained in the selected
	intervals are counted exactly by
	\begin{equation}\label{eq:bridgepacking}
		r=(a_1-1)
		+\sum_{j=1}^{k-1}(a_{j+1}-a_j-D)
		+(n-a_k-D).
	\end{equation}
	Every term on the right is nonnegative. In particular, every gap with
	$a_{j+1}-a_j>D$ contributes at least one to the middle sum. Hence there
	are at most $r$ such gaps.
	
	Let $s_0$ be the number of indices $j$ for which
	$a_{j+1}-a_j=D$. Then
	$s_0\ge k-1-r$. Since $r=n-1-Dk$, this gives
	\[
	n-Ds_0
	\le n-D(k-1-r)
	=(D+1)(r+1)
	\le C_d\eta n,
	\]
	after increasing $C_d$ if necessary. Thus the $D$-vertex sets lying
	between consecutive bridge cuts at distance $D$ cover all but at most
	$C_d\eta n$ vertices. If $k<2$, the same conclusion follows after
	increasing $C_d$, so we may assume that such a pair of cuts exists.
	
	Now consider two cuts of size one at positions $a$ and $a+D$, and put
	$T=\{a+1,\ldots,a+D\}$. Since $G$ is connected, an edge forming a cut
	of size one is a bridge. The two bridges are distinct. Indeed, if they
	were the same edge, that edge would have one endpoint before $T$ and
	the other after $T$, and no edge would join $T$ to the rest of the
	graph. Similarly, no edge can join a vertex before $T$ directly to a
	vertex after $T$. Hence exactly two edges leave $T$, one through each
	of the two cuts.
	
	Since every vertex has degree at least $d$,
	\[
	2e(T)+2\ge dD.
	\]
	As $D=d+1$, this gives
	\[
	e(T)\ge\frac{dD-2}{2}
	=\binom D2-1.
	\]
	Because $G$ is simple, $e(T)\le\binom D2$. Therefore $G[T]$ is either
	$K_D-e$ or $K_D$.
	
	Suppose $G[T]=K_D-e$. The two endpoints of the missing edge have
	internal degree $D-2=d-1$. Since every vertex has degree at least $d$,
	each of them must have a neighbor outside $T$. There are exactly two
	edges leaving $T$, so these are precisely the two edges leaving the block, one
	at each endpoint of the missing edge. Both are bridges. Thus $G[T]$ is
	a copy of $L_d$ with the required two attachments. Since $L_d$ has
	no cut vertex and its only exits are bridges, this copy is a block
	of $G$.
\end{proof}

\begin{lemma}\label{lem:completegaps}
	Among the intervals in \Cref{lem:bridgeblocks}, those inducing
	$K_D$ contain altogether at most $C_d\eta n$ vertices. Consequently,
	there are vertex-disjoint copies of $L_d$ covering all but
	$C_d\eta n$ vertices of $G$.
\end{lemma}

\begin{proof}
	Consider a gap between two bridge cuts at positions $a$ and $a+D$
	whose $D$ vertices induce $K_D$. Let
	\[
	B=B_a=\{a,\ldots,a+D-1\}
	\]
	be the selected interval beginning at the left bridge. Write
	$w_s=a_{a+s}$ for $0\le s\le d$, and put
	$Z=w_1+\cdots+w_d$. Thus
	$	\Delta_B=w_0+Z.$	
	All internal edges of the induced $K_D$ have both difference indices
	inside $B$, while the incoming bridge contributes $w_0^2$. Applying
	\eqref{eq:cliquesquares} to the clique and dropping the nonnegative
	square terms gives
	\[
	Q_B\ge w_0^2+\frac D2Z^2.
	\]
	By weighted Cauchy--Schwarz, $w_0^2+DZ^2/2\ge D(w_0+Z)^2/(D+2).$
	Hence
	\[	Q_B\ge\frac{D}{D+2}\Delta_B^2.\]
	Since $\y$ is obtained by linear interpolation on $B$, all the
	differences $b_i$, $i\in B$, are equal to $\Delta_B/D$. Therefore
	\[
	\sum_{i\in B}b_i^2
	=D\left(\frac{\Delta_B}{D}\right)^2
	=\frac{\Delta_B^2}{D}.
	\]
	Recall that
	\[
	R_B=Q_B-\al\frac{\Delta_B^2}{D}.
	\]
	Using the lower bound for $Q_B$, we obtain
	\begin{align}
		R_B
		&\ge
		\left(\frac{D}{D+2}-\frac{\al}{D}\right)\Delta_B^2 \notag\\
		&=
		\left(\frac{D^2}{D+2}-\al\right)
		\sum_{i\in B}b_i^2.
	\end{align}
	Since $\al=D-2$, this becomes
	\begin{equation}\label{eq:completegap}
		R_B\ge\frac4{D+2}\sum_{i\in B}b_i^2.
	\end{equation}
	
	Let $\mathcal B_0$ be the family of selected intervals corresponding
	to the $K_D$ gaps, and put
	$A_0=\bigcup_{B\in\mathcal B_0}B.$
	The intervals in $\mathcal B_0$ are disjoint. Summing
	\eqref{eq:completegap} over $B\in\mathcal B_0$ therefore gives
	\[	\frac4{D+2}\sum_{i\in A_0}b_i^2
	\le \sum_{B\in\mathcal B_0}R_B.	\]
	By \eqref{eq:excess}, all terms in the decomposition of $R$ are
	nonnegative, and hence
$\sum_{B\in\mathcal B_0}R_B
	\le \sum_B R_B
	\le R.$
	Consequently, by \eqref{eq:smallR},
	\[	\sum_{i\in A_0}b_i^2
	\le \frac{D+2}{4}R
	\le C_d\frac{\theta}{n^2}.	\]
		\Cref{lem:pathstable}, applied with $A=A_0$, now gives
	\[
	|A_0|\le C_d\theta^{1/3}n=C_d\eta n.
	\]
	Each selected interval in $A_0$ contains $D$ difference indices, and
	the corresponding $K_D$ gap contains exactly $D$ vertices. Hence the
	vertices lying in all $K_D$ gaps number at most $C_d\eta n$.
	
	By \Cref{lem:bridgeblocks}, apart from at most $C_d\eta n$
	vertices, every vertex lies in a gap between consecutive bridge cuts
	at distance $D$. Each such gap induces either $K_D$ or $L_d=K_D-e$.
	We have just shown that the $K_D$ gaps contain only
	$C_d\eta n$ vertices. Therefore, after increasing $C_d$ if necessary,
	the remaining $L_d$ blocks cover all but $C_d\eta n$ vertices.
\end{proof}

\begin{proof}[Proof of \Cref{thm:stability}]
	If $\theta=\e+n^{-1}$ is bounded away from zero, both numerical
	bounds are trivial after increasing $C_d$; take the selected family
	to be empty. We may therefore assume $\theta$ is sufficiently small
	that the coverage estimate below gives at least one selected block.
	By \Cref{lem:completegaps}, there are vertex-disjoint copies of
	$L_d$ covering all but $C_d\eta n$ vertices. Since
	$\eta=(\e+n^{-1})^{1/3}$, this proves \eqref{eq:coverage}.
	
	It remains to describe how these blocks are arranged. The cuts at
	$a_1,\ldots,a_k$ are nested because they are cuts between the prefixes
	$\{1,\ldots,a_j\}$ and their complements. Moreover, both sides of every
	cut of size one are connected. Indeed, if one side had two components,
	each component would need an edge to the other side in order for $G$ to
	be connected, giving at least two edges in the cut.
	
	Hence the bridge cuts occur in one linear order. Every selected
	$L_d$ block lies between two consecutive bridges in this order.
	By \Cref{lem:bridgeblocks}, exactly two edges leave such a block,
	one at each endpoint of its missing edge, and both are bridges.
	Therefore the selected blocks occur along chains of consecutive
	bridges, with no additional edges leaving a selected block.
	
	Finally, let $s_1$ be the number of selected $L_d$ blocks. Since they
	cover all but $C_d\eta n$ vertices,
	\[
	Ds_1\ge n-C_d\eta n.
	\]
	Consider a path from the left attachment vertex of the first selected
	block to the right attachment vertex of the last. In every copy of
	$L_d$, the two attachment vertices are the endpoints of the missing
	edge and are therefore nonadjacent. Any path through the block needs
	at least two internal edges. Between two consecutive blocks, at least
	one further edge is needed. Thus the path has length at least
	$2s_1+(s_1-1)=3s_1-1$.
	
	Consequently,
	\[
	\operatorname{diam}(G)
	\ge 3s_1-1
	\ge \frac{3n}{D}-C_d\eta n.
	\]
	Since $D=d+1$, this is \eqref{eq:diamstable}. The proof of
	\Cref{thm:stability} is complete.
\end{proof}


\section{The structure of \texorpdfstring{$\mu$}{mu}-minimal graphs}\label{sec:GMfull}
In this section, we prove \Cref{thm:GMfull} in the minimum-degree
and odd-regular classes. By \Cref{lem:comparisonupper} and exact
minimality, a minimizing graph satisfies \eqref{eq:comparisonerror}.
Applying \Cref{thm:stability} leaves $o(n)$ vertices outside the
prescribed blocks. We first show that every consecutive exceptional
subgraph has bounded order, then that these subgraphs lie near the
ends, and finally that the whole block-tree is a path.

The regular case uses comparison graphs with the prescribed degrees.
A $d$-regular minimizer need not minimize algebraic connectivity in
the larger class $\delta\ge d$, so its admissible replacements and
switches are checked separately.

Fix $d\ge3$ and choose the class $\delta\ge d$, or, when $d$ is odd,
the $d$-regular class. A minimizer in this section means an exact
minimizer in the chosen class. Every change made in the regular class
preserves all degrees. The minimum-degree argument also covers every
minimizer subject to $\delta=d$. Indeed, start with a graph of minimum
degree greater than $d$ and delete edges on cycles until its minimum
degree first becomes $d$. Before that stage a cycle exists, since
the minimum degree is at least $d+1$. Deleting a cycle edge preserves
connectedness and does not increase $\mu$. Thus the classes
$\delta=d$ and $\delta\ge d$ have the same minimum at each order,
and every minimizer in the former is a minimizer in the latter.


For a subgraph $H$, write
\[\E_H(\x)=\sum_{uv\in E(H)}(x_u-x_v)^2.\]
In particular, $\E_G(\x)=\x^\top L(G)\x$. In the comparisons below,
the trial vector need not have mean zero, so we use
\eqref{eq:centeredRayleigh} in the form
\[\mu(G)\le\frac{\E_G(\x)}{\|\x-\bar x\1\|^2}.\]
Thus the local changes will be measured through the corresponding
quantities $\E_H(\x)$, while the full numerator is $\E_G(\x)$.
A replacement is called \emph{admissible} if, after the indicated
boundary edges are reattached, the resulting graph is simple,
connected, has the same order, and belongs to the chosen degree class.

\subsection{An ordered bridge and its endpoints}
We show where the endpoints of a bridge lie in the vertex order.
This will help us compare the pieces on either side.


Let $\x$ be a mean-zero Fiedler vector, with its components ordered as
$x_1\le\cdots\le x_n$. If a cut between two consecutive positions has
exactly one crossing edge, we call this edge an {\em ordered bridge}.
Both sides of such a cut are connected: otherwise one side would
require at least two outgoing edges.
Suppose the cut is at position $i$, and let $uv$ be its unique crossing
edge, where $u\le i<v$. Summing the eigen-equations over
$\{1,\ldots,i\}$ gives
\begin{equation}\label{eq:prefixsum}
	x_v-x_u
	=-\mu\sum_{j=1}^i x_j
	=\mu f_i,\qquad f_i=-\sum_{j=1}^i x_j.
\end{equation}
By \eqref{eq:AFprefix}, every proper prefix has negative sum, so $f_i>0$. Hence
$x_v-x_u>0$.

We use the usual two-edge switch, also used in
\cite[Section~3.1]{AbGhIm}. For distinct vertices $u,v,w,z$,
replacing $uv,wz$ by $wv,uz$ preserves every degree. If the added
edges are absent, the resulting graph is simple. Its change in the quadratic form is
\begin{equation}\label{eq:switch}
 (x_w-x_v)^2+(x_u-x_z)^2-(x_u-x_v)^2-(x_w-x_z)^2
       =2(x_w-x_u)(x_z-x_v).
\end{equation}
Connectedness is checked each time this operation is used.

\begin{lemma}\label{gm:lem:bridgeextrema}
Let $uv$ be an ordered bridge of a minimizer, with $u$ on its left
side $S$ and $v$ on its right side. Then
\[
 x_u=\max_{w\in S}x_w,\qquad x_v=\min_{w\notin S}x_w.
\]
In particular, if the cut is after position $i$, then
$x_i=x_u<x_v=x_{i+1}$.
\end{lemma}
\begin{proof}
Suppose $w\in S$ has $x_w>x_u$. If $\deg(u)>d$, replace $uv$ by
$wv$. This case occurs only in the minimum-degree class. The new
graph is simple, connected and has minimum degree at least $d$.
Since $x_u<x_w\le x_v$, its quadratic form strictly decreases.

It remains that $\deg(u)=d$. We choose
$z\in N(w)\cap S$, distinct from $u$, with $uz\notin E(G)$, so
that $u,w$ remain connected in $G[S]-wz$.
If $uw$ is an edge, $w$ has at least $d-1$ further neighbours in
$S$, whereas $u$ has only $d-2$ neighbours there other than $w$.
The required $z$ exists and $uw$ is retained.
If $u,w$ are nonadjacent but have a common neighbour, their degrees
within $S$ are $d-1$ and at least $d$. Choose
$z\in N(w)\setminus N(u)$; the path through a common neighbour
is retained. Finally, if there is no common neighbour, every
neighbour of $w$ is eligible. Choose one other than the first
neighbour on a shortest $w$--$u$ path in $S$.

Replace $uv,wz$ by $wv,uz$. Both added edges are absent and all
degrees are unchanged. If deleting $wz$ disconnects $S$, its
component containing $z$ is reattached by $uz$, while the component
containing $u,w$ is joined to the right side by $wv$.
Thus the new graph is connected. Equation~\eqref{eq:switch}
gives a nonpositive change in the quadratic form, since $x_w>x_u$ and $x_z\le x_v$.
A strict inequality contradicts minimality. In the equality case
the eigen-equation at $v$ changes by $x_u-x_w\ne0$, so equality
in the Rayleigh principle is impossible as well. Reversal proves
the other assertion, and \eqref{eq:prefixsum} gives strictness.
\end{proof}


The vertices between two consecutive ordered bridges form a connected subgraph, and exactly two edges leave it, namely the two boundary bridges. These two bridges are different; otherwise one edge would cross both boundary cuts, leaving the vertices between them disconnected from the rest of the graph.

\begin{lemma}\label{gm:lem:attachments}
An interior portion bounded by two ordered bridges of a minimizer
has distinct attachment vertices.
\end{lemma}
\begin{proof}
Suppose both bridges meet the same vertex $r$ of the portion $H$.
\Cref{gm:lem:bridgeextrema} makes every component on $H$
equal to $x_r$. There is another vertex in $H$, since otherwise
$r$ has degree two. Its eigen-equation gives $\mu x_r=0$, so
$\x_H=\0$. The two outside attachment values have opposite strict
signs. Write $rv$ for the right bridge, with $x_v>0$, and choose
$w\in N_H(r)$.

If $\deg(r)>d$, replace $rv$ by $wv$. This preserves simplicity,
connectedness and minimum degree. It occurs only in the
minimum-degree class. The quadratic form of $\x$ is unchanged, but its
eigen-equation at $w$ fails, a contradiction.

If $\deg(r)=d$, then
$|N_H(w)\setminus\{r\}|\ge d-1$, whereas
$|N_H(r)\setminus\{w\}|=d-3$. Choose
\[
 z\in N_H(w)\setminus\bigl(N_H(r)\cup\{r\}\bigr)
\]
and replace $rv,wz$ by $wv,rz$. Both added edges are absent and
every degree is preserved. The edge $rw$ is retained. If deleting
$wz$ disconnects $H$, the component containing $z$ is reattached
to the component containing $r,w$ by $rz$. The right outside part
is attached by $wv$. Thus the new graph is connected.

All components on $H$ are zero, so the mean, norm and quadratic form of
$\x$ are unchanged. At $w$, however, the corresponding component of the new Laplacian applied to $\x$ is
$-x_v\ne0=\mu x_w$. Equality in the Rayleigh principle is impossible,
again contradicting minimality.
\end{proof}


Now that the attachment vertices are known to be distinct, we may
permute tied components so that the left and right attachments of
each interior portion come first and last within that portion.
By \Cref{gm:lem:bridgeextrema}, no ordered bridge cut lies inside a
set of tied components, before or after this permutation. Thus the
positions of all ordered bridges are unchanged.

\subsection{Effective resistance between two vertices}
We find a bound for comparing graph pieces that have the same two
attachment vertices.

For a connected graph $H$ and distinct vertices $s,t$, write
\[ R_H(s,t)=\left(\min_{z_s=0,z_t=1}\E_H(\z)\right)^{-1},\]
which is indeed the \emph{effective resistance} between $s$ and $t$.
Equivalently, $R_H(s,t)$ is the smallest constant $R$ such that
$ (z_t-z_s)^2\le R\,\E_H(\z) $
for every real vector $\z$ on $V(H)$.
In this subsection we estimate the effective resistance between two specified
vertices of a graph. These estimates will later be used to compare different
parts of the graph while keeping the same two attachment vertices.

Put $D=d+1$, $\alpha=d-1$, and $L_d=K_D-e$.

For disjoint nonempty vertex sets $A,B$, write $\bar z_A,\bar z_B$
for their average values. If all edges between the two sets are present,
their contribution is
\begin{align}\label{eq:clique-averaging}
 \sum_{u\in A,\,v\in B}(z_u-z_v)^2
 ={}&|A||B|(\bar z_A-\bar z_B)^2\notag\\
 &+|B|\sum_{u\in A}(z_u-\bar z_A)^2
  +|A|\sum_{v\in B}(z_v-\bar z_B)^2.
\end{align}
Indeed, expand $z_u-z_v$ about the two averages; the mixed sums
vanish. Consequently, averaging within a clique of vertices with
the same outside neighbors cannot increase the quadratic form.
The displayed remainders also record the equality conditions.

For $L_d$, the two singleton endpoints have the same $d-1$ neighbors.
With endpoint values $0,\Delta$, averaging on that middle clique and
then applying \eqref{eq:weighted-squares} to its two sets of
$d-1$ connecting edges gives the unique minimizing middle value
$\Delta/2$ and minimum $(d-1)\Delta^2/2$. Thus
$R_{L_d}=2/(d-1)$. This is an instance of the Dirichlet principle;
see \cite[Exercise~1.3.11]{DoyleSnell}.

\begin{theorem}\label{gm:thm:terminal}
Let $H$ be a connected  graph of order $m$, with distinct
specified vertices $s,t$. Suppose $\deg_H(s),\deg_H(t)\ge d-1$ and
every other vertex has degree at least $d$. Then
\begin{equation}\label{gm:eq:terminal}
 R_H(s,t)\le\frac{m}{d-1}-1.
\end{equation}
Equality holds precisely when $m$ is a multiple of $D$ and $H$ is a chain of copies of
$L_d$, joined by bridges, with $s,t$ the two unattached missing-edge endpoints.
\end{theorem}
\begin{proof}
The degree assumptions imply $m\ge D$. Choose the minimizing vector, order its components
increasingly, and put $s$ first and $t$ last. At a vertex other than the two specified
vertices, a minimizing component is the average of its neighbours. Thus all components lie
between the two specified values. 
With $a_i=z_{i+1}-z_i$, equation~\eqref{eq:full} gives
\[\E_H(\z)=\sum_{i,j}K_{ij}a_i a_j,\] where $K_{ij}$ is the number of edges of $H$ crossing both cuts
at positions $i$ and $j$.


Put $\Delta_0=\sum_{i=1}^d a_i$, and set $v_j=z_j-z_s$ for
$1\le j\le D$. Define $Q_0=\sum_{i,j=1}^dK_{ij}a_i a_j$.
For $1\le u\le v\le d$, counting edges from the first $u$ vertices
gives
\[
 K_{uv}\ge du-1-u(u-1)-u(v-u)=u(D-v)-1.
\]
Only $s$ can contribute a degree deficiency in this set; the other
specified vertex is last. Thus the same comparison as in
\eqref{eq:cliquecomparison} gives
\[Q_0\ge \bv^\top L(K_D)\bv-\Delta_0^2.\]
By \eqref{eq:cliquesquares},
$\bv^\top L(K_D)\bv\ge D\Delta_0^2/2$. Hence
\[Q_0\ge
\left(\frac D2-1\right)\Delta_0^2
=\frac{\al}{2}\Delta_0^2.\]
We also need a lower bound on the cuts near the two ends. For
$1\le i\le d$, degree counting gives
$q_i\ge i(D-i)-1$. Since the minimum of $i(D-i)-1$ on this range is
attained at $i=1$ or $i=d$, we have
$q_i\ge D-2=\al.$
The same argument, applied from the other end, gives the same bound
for the last $d$ cuts.

For the remaining cuts use the intervals of \eqref{eq:windows}.
The vertices counted between two internal cuts exclude $s,t$, so
\eqref{eq:cutcount} and \eqref{eq:overlap} hold there without a degree
deficiency. The endpoint cut bounds above ensure that the selected
intervals are disjoint from the first $d$ differences and fit before
the final endpoint. \Cref{lem:local} consequently applies without
change. Thus a bridge interval has $Q\ge\alpha\Delta^2/D$, another
selected interval has $Q\ge(\alpha+\kappa_d)\Delta^2/d$, and every
uncovered cut has size at least $\alpha$. We shall also use the
uncompleted bridge estimate from that proof:
\begin{equation}\label{gm:eq:terminalbridge}
 Q\ge a_a^2+\frac\alpha2 Z^2+2\sum_j(v_j-Z/2)^2,
 \qquad Z=a_{a+1}+\cdots+a_{a+d},
\end{equation}
where the $v_j$ are the interior accumulated differences.

Keeping all the indicated terms and applying weighted Cauchy--Schwarz gives
\[
 \E_H(\z)\ge\frac{(z_t-z_s)^2}{2/\alpha+(m-1-d)/\alpha}
 =\frac{(z_t-z_s)^2}{m/\alpha-1},
\]
which proves the inequality.

We now consider the equality case. Let $\z$ minimize $\E_H(\z)$ subject
to $z_s=0$ and $z_t=1$. By the definition of effective resistance,
$\E_H(\z)=1/R_H(s,t)$. Put
\[
\widetilde{\z}=R_H(s,t)\z.
\]
Then
$\widetilde z_t-\widetilde z_s=R_H(s,t)$ and
$\E_H(\widetilde{\z})=R_H(s,t)$.
The minimizing vector is harmonic at every vertex other than $s$ and
$t$. For the original normalization $z_t-z_s=1$, the total difference
along the edges leaving $s$ is $1/R_H(s,t)$. Hence, after multiplying
by $R_H(s,t)$, the total difference along the edges leaving $s$ is
equal to $1$. Summing the harmonic equations over the vertices on the
left of any ordered cut cancels all contributions from edges lying
entirely on that side. Thus, for $\widetilde{\z}$, the sum of the
differences over the edges crossing every ordered cut is also $1$.

Suppose now that equality holds in the resistance bound proved above.
Then equality must hold at every step used in its proof. In particular,
we must have equality in the weighted Cauchy--Schwarz inequality.
The equality condition there says that the total differences of the
pieces are proportional to the reciprocals of their coefficients.
The total endpoint difference of $\widetilde{\z}$ equals
$R_H(s,t)=m/\alpha-1$, which is also the sum of those reciprocal
coefficients. Hence the proportionality constant is one. This gives
a total difference $2/\alpha$ on the initial interval,
$|B|/\alpha$ on every selected interval $B$, and $1/\alpha$ on every
uncovered index. In particular, every selected interval has positive
total difference.
A nonbridge selected interval cannot occur. Indeed, the local
comparison for such an interval is strict whenever its total difference
is positive. Thus every selected interval must begin at a bridge.
We next use the equality conditions inside a bridge interval.
Formula~\eqref{gm:eq:terminalbridge} shows that its two end differences
are $1$ and $1/\alpha$, respectively. For the initial interval, the
corresponding equality condition gives an end difference $1/\alpha$
at each end. In particular, all these end differences are positive.
There can be no uncovered index. Suppose that $i$ were uncovered.
As observed above, its difference is $1/\alpha>0$. Any neighboring
difference index is either also uncovered or is an end index of the
initial interval or of a bridge interval, and in every case its
difference is positive. Equation~\eqref{eq:neighborcount}, including
the two end cases, then gives a positive term among the terms that were
discarded in obtaining the lower bound. Equality would therefore be
impossible. Hence there are no uncovered indices.
It follows that the first $d$ differences and the selected intervals
of length $D$ cover all the differences, without overlap. Since
$d=D-1$, the number of vertices $m$ is therefore a multiple of $D$.
Moreover, the bridges occur between consecutive groups of $D$ vertices.

Consider one such group $T$. Let $b_T$ be the number of edges joining
$T$ to the rest of $H$, and let
$h_T=|T\cap\{s,t\}|$. For an interior group we have $b_T=2$ and
$h_T=0$; for an end group we have $b_T=1$ and $h_T=1$; and if there
is only one group, then $b_T=0$ and $h_T=2$. Thus in every case
\[
b_T+h_T=2.
\]
Using the degree conditions and summing the degrees of the vertices of
$T$, we obtain
\[
2e(T)\ge dD-h_T-b_T=dD-2.
\]
Since $D=d+1$, this gives
$e(T)\ge\binom D2-1$. As $H$ is simple, every group therefore induces
either $K_D$ or $K_D-e$.

The case $K_D$ cannot occur in the equality case. On the initial
vertex group, the actual internal quadratic form would then be
$\bv^\top L(K_D)\bv$, whereas the comparison used
$\bv^\top L(K_D)\bv-Z^2$. Here the total internal difference is
$Z=2/\alpha>0$. The same strict improvement by $Z^2$ applies to
any later group, with the incoming bridge term retained separately
in \eqref{gm:eq:terminalbridge}. Thus a complete group would make
the resistance inequality strict.
Hence every group induces $K_D-e$. The two endpoints of its missing
edge have internal degree $d-1$. The degree conditions therefore force
the bridge endpoints, or $s$ and $t$ at the two ends of the chain, to
be precisely these two vertices. Thus $H$ has exactly the chain
structure stated in the theorem.
Conversely, every graph with this chain structure attains equality.
Each copy of $K_D-e$ has effective resistance $2/\alpha$ between its
two attachment vertices, and every bridge has resistance $1$. If there
are $r$ copies, then $m=rD$, and the resistances add in series:
\[
R_H(s,t)
=r\frac{2}{\alpha}+(r-1)
=r\frac{D}{\alpha}-1
=\frac{m}{\alpha}-1.
\]
Hence all the stated chains give equality.
\end{proof}

\subsection{Replacing an end subgraph}
In this subsection, we compare end subgraphs through their contribution
at the attachment vertex. This avoids changing the mean of a trial vector.

We use one matrix identity for fixed boundary values. Let
\[
 M=\begin{pmatrix}C&-J\\-J^\top&B\end{pmatrix},\qquad
 S=C-JB^{-1}J^\top,
\]
where $M$ is real symmetric and $B$ is positive definite.

Direct expansion gives
\begin{equation}\label{eq:boundary-square}
 \begin{pmatrix}\bu\\\bv\end{pmatrix}^{\!\top}
 M\begin{pmatrix}\bu\\\bv\end{pmatrix}
 =\bu^\top S\bu+
 (\bv-B^{-1}J^\top\bu)^\top B(\bv-B^{-1}J^\top\bu).
\end{equation}
For fixed $\bu$, the minimum is therefore $\bu^\top S\bu$, attained
at $\bv=B^{-1}J^\top\bu$. The matrix $S$ is the \emph{Schur
complement}; see \cite[Section~12.7]{Spielman}.
Moreover, $M$ and $S$ have the same number of negative eigenvalues,
counted with multiplicity. To see this, the invertible change
$(\bu,\bv)\mapsto(\bu,\bv-B^{-1}J^\top\bu)$ makes the quadratic
form block diagonal with blocks $S,B$. An invertible change
preserves the largest dimension of a subspace on which a quadratic
form is negative definite. By diagonalization, this dimension is
its number of negative eigenvalues, and $B$ has none. This is the
case of Sylvester's law of inertia needed here; see
\cite[Theorem~24.2.2]{Spielman}.

We will also use its exact small-parameter form. Suppose a quadratic
form with fixed boundary values has minimum $q(0)$ at the vector
$\mathbf f_0$ of free values, and its excess at $\mathbf f_0+\mathbf w$
is $\mathbf w^\top B\mathbf w$. Let $\omega$ be the fixed boundary
contribution to the squared norm. If $B-tI$ is positive definite,
the minimum after subtracting $t$ times the squared norm is
\begin{equation}\label{eq:boundary-expansion}
 q(t)=q(0)-t(\|\mathbf f_0\|^2+\omega)
             -t^2\mathbf f_0^\top(B-tI)^{-1}\mathbf f_0.
\end{equation}
Indeed, the terms depending on $\mathbf w$ are
$\mathbf w^\top(B-tI)\mathbf w-2t\mathbf f_0^\top\mathbf w$.
Their minimum is the last term in \eqref{eq:boundary-expansion},
by \eqref{eq:boundary-square}. In particular, the minimizing
change is $\mathbf w=t(B-tI)^{-1}\mathbf f_0$.

An end subgraph is a connected subgraph joined to the rest by one
bridge. If $r$ is its attachment vertex, its internal degree is at
least $d-1$, and all other internal degrees are at least $d$.
The order of this portion is not yet known to be bounded. The
following estimates keep their dependence on that order explicit.
Let $H$ be connected, of order $b$, with specified vertex $r$. Attach it to the rest of a graph by one bridge. Put
\[
 A_H=L(H)+e_re_r^\top,\quad
 F_H(t)=e_r^\top(A_H-tI)^{-1}e_r,\quad
 T_H(t)=\1^\top(A_H-tI)^{-1}\1.
\]
\begin{lemma}\label{gm:lem:response}
The smallest eigenvalue of $A_H$ is at least $b^{-2}$, and
\[
 F_H(t)=1+bt+t^2T_H(t)\qquad
 (0\le t<\lambda_{\min}(A_H)).
\]
For $0\le t<b^{-2}$,
\[
 T_H(0)\le T_H(t)\le\frac{T_H(0)}{1-tb^2},\qquad T_H(0)\le b^3.
\]
When $tb^2\le1/2$, one also has
\begin{equation}\label{eq:resolventbounds}
 \|(A_H-tI)^{-1}\|\le2b^2,\quad
 \|(A_H-tI)^{-1}\1\|^2\le4b^5,\quad
 0\le T_H(t)-T_H(0)\le2tb^5.
\end{equation}
The norm of a matrix here is its Euclidean operator norm.
\end{lemma}
\begin{proof}
Add a zero-valued vertex $0$ joined to $r$. Every vertex of $H$
has a path to $0$ of at most $b$ edges. Cauchy--Schwarz on this
path gives $z_v^2\le b\,\z^\top A_H\z$. Summing over the $b$
vertices proves $\|\z\|^2\le b^2\z^\top A_H\z$ and hence the
eigenvalue bound.

For $0\le t<b^{-2}$ and an eigenvalue $\lambda\ge b^{-2}$,
\[
 \frac1\lambda\le\frac1{\lambda-t}
 \le\frac1{1-tb^2}\frac1\lambda.
\]
Expanding $\1$ in an orthonormal eigenbasis proves the bounds for
$T_H(t)$. Also $T_H(0)\le\|\1\|^2\|A_H^{-1}\|\le b^3$.
If $tb^2\le1/2$, then $\|(A_H-tI)^{-1}\|\le2b^2$, giving the
first two estimates in \eqref{eq:resolventbounds}. The identity
\[
 (A_H-tI)^{-1}-A_H^{-1}
 =tA_H^{-1}(A_H-tI)^{-1}
\]
gives the last estimate, since $\|\1\|^2=b$.

Finally, $A_H\1=e_r$, so
$(A_H-tI)^{-1}e_r=\1+t\bu$, where
$\bu=(A_H-tI)^{-1}\1$. Summing $(A_H-tI)\bu=\1$ gives
$u_r-tT_H(t)=b$. Taking the $r$-component of the preceding
inverse identity yields
$F_H(t)=1+tu_r=1+bt+t^2T_H(t)$.
\end{proof}
The next lemma gives a strict comparison in terms of the attachment
response. It is related to the bounded-perturbation comparison in
\cite[Theorem~2.4]{AbGhDiam}, but does not require an upper bound on
the increase of the response.
\begin{lemma}\label{gm:lem:replacecap}
Let $\mu=\mu(G)$, and suppose a Fiedler vector has nonzero value
at the vertex immediately outside an end subgraph $H$.
Let $H'$ be an admissible same-order replacement, and suppose
\[
 \mu<\min\{\lambda_{\min}(A_H),\lambda_{\min}(A_{H'})\}.
\]
If $\Delta=F_{H'}(\mu)-F_H(\mu)>0$, then the replaced graph $G'$
has smaller algebraic connectivity.
\end{lemma}
\begin{proof}
Let $s$ be the vertex immediately outside $H$, and let $r$ be
its attachment vertex in $H$. Order the exterior vertices first.
Then
\[
 L(G)-\mu I=
 \begin{pmatrix}C&-e_se_r^\top\\-e_re_s^\top&A_H-\mu I\end{pmatrix}.
\]
The exterior block $C$ is unchanged by the replacement, since the
exterior graph and its single bridge to the end subgraph are unchanged.
By \eqref{eq:boundary-square}, the reduced exterior matrix is
$S=C-F_H(\mu)e_se_s^\top$.
The matrix $L(G)-\mu I$ has exactly one negative eigenvalue:
$G$ is connected, and $\mu$ is its second Laplacian eigenvalue.
Thus $S$ also has exactly one negative eigenvalue.

Let $\z$ be the exterior restriction of the Fiedler vector.
Eliminating its interior values gives $S\z=\0$, and $z_s\ne0$
by hypothesis. The new reduced matrix is
$S'=S-\Delta e_se_s^\top$. Choose a unit eigenvector $\bu$ of
$S$ with eigenvalue $-\sigma<0$. Since $S\z=\0$, we have
$\bu\perp\z$. For real $\xi,\zeta$,
\begin{equation}\label{gm:eq:replacechange}
 (\xi\bu+\zeta\z)^\top S'(\xi\bu+\zeta\z)
 =-\sigma\xi^2-\Delta(\xi u_s+\zeta z_s)^2<0
 \quad\text{if }(\xi,\zeta)\ne(0,0).
\end{equation}
Indeed, the first term is negative if $\xi\ne0$; otherwise
use $\Delta>0$ and $z_s\ne0$. Hence $S'$ has at least two
negative eigenvalues. The positive definiteness of $A_{H'}-\mu I$
and \eqref{eq:boundary-square} give the same conclusion for
$L(G')-\mu I$. Therefore $\mu(G')<\mu(G)$.
\end{proof}
\subsection{Bounding the exceptional gaps}
We show that positions outside the desired bridge intervals have a
definite cost in our estimates.

We now estimate how many difference indices fail to lie in the desired bridge intervals. We group together the uncovered indices and the selected intervals that do not begin at a bridge, and show that their total contribution is small. This will imply that only a small number of positions belong to exceptional gaps.

All constants in this section depend only on $d$. We use the nondecreasing order of the
vector under discussion. A selected interval starts at the first cut of a group of cuts of
size at most $d-2$. Its length is $D$ if the first cut has size one, and $d$ otherwise.
Call these bridge intervals and other selected intervals, respectively. Write $U$ for the
uncovered difference indices and $V$ for the union of $U$ and all other selected
intervals.

\begin{lemma}\label{gm:lem:fixedloss}
Let $H,s,t$ satisfy \Cref{gm:thm:terminal}, and let $z_s=\min \z<\max \z=z_t$. The
first $d$ differences are treated separately as in that theorem; select the subsequent
intervals as above. There is $c_d>0$ such that, for every $I>0$,
\begin{equation}\label{gm:eq:fixedloss}
 \left(\frac m\alpha-1\right)I^2-2I(z_t-z_s)+\E_H(\z)
 \ge c_d|V|I^2.
\end{equation}
In particular, if every gap between consecutive cuts of size one has more than $D$
vertices, then, for $m$ sufficiently large,
\begin{equation}\label{gm:eq:strongcell}
 \E_H(\z)\ge \frac{\alpha+\eta_d}{m}(z_t-z_s)^2
\end{equation}
for a constant $\eta_d>0$.
\end{lemma}
\begin{proof}
Put $a_i=z_{i+1}-z_i$. Retain all terms within the initial interval and within each
selected interval, as well as the uncovered diagonal terms. All other terms are
nonnegative. Complete squares after subtracting $2I\sum a_i$.
The constants sum to
$[2+(m-1-d)]I^2/\alpha=(m/\alpha-1)I^2$; thus the total
remainder $\mathcal R$ is the left side of \eqref{gm:eq:fixedloss}.
The initial interval uses
the constant $2I^2/\alpha$. An interval beginning at a bridge uses $DI^2/\alpha$ and has
remainder at least
\begin{equation}\label{gm:eq:bridgeremainder}
 (a_a-I)^2+\frac\alpha2(Z-2I/\alpha)^2
       +2\sum_j(v_j-Z/2)^2.
\end{equation}
In particular its first and last differences are controlled, in squared distance, from $I$
and $I/\alpha$. For the initial interval, if $Z$ is its total difference and $v_j$ its
accumulated interior components, the remainder is at least
\[
 \frac\alpha2(Z-2I/\alpha)^2+2\sum_j(v_j-Z/2)^2.
\]
Its first difference is $v_1$ and its last is $Z-v_{d-1}$, so both are controlled in
squared distance from $I/\alpha$. The sum of these endpoint squared distances is at most a
constant depending on $d$ times $\mathcal R$.

An uncovered index contributes at least $\alpha(a_i-I/\alpha)^2$. Any other selected
interval has the strict bound $Q\ge(\alpha+\kappa_d)\Delta^2/d$ from
\Cref{lem:local}; consequently its remainder is at least
\[
 h_d I^2,\qquad h_d=d\left(\frac1\alpha-\frac1{\alpha+\kappa_d}\right)>0.
\]
There are no such intervals when $d=3$.

Apply \Cref{lem:adjacent} with the constant values
$p_i=I/\alpha$. Take $c_j=\alpha$ at the first index of a bridge
interval and $c_j=1$ at its last index and at both endpoints of
the initial interval. The preceding squares control the endpoint
and uncovered errors. The strict bound $h_dI^2$ controls
$\sum_{i\in A}p_i^2$, since each nonbridge interval has length $d$.
All adjacent products having an uncovered index are among the
remaining nonnegative cross terms. Also $q_i\ge\alpha\ge2$ on $U$.
The variation term in \eqref{eq:adjacentcontrol} is zero, giving
$|V|I^2/\alpha^2\le C_d\mathcal R$. This is
\eqref{gm:eq:fixedloss}.

If there are $k$ cuts of size one, their mutual distances are at least $D+1$ by
hypothesis. Consequently $kD\le Dm/(D+1)+C_d$. Apart from the initial $d$ differences, the
indices outside their bridge intervals are exactly $V$. Hence $|V|\ge m/(D+1)-C_d$.
Inequality \eqref{gm:eq:fixedloss} reduces the coefficient of $I^2$ in the upper bound for
$2I(z_t-z_s)-\E_H(\z)$ from $m/\alpha$ to $(1/\alpha-c'_d)m+O_d(1)$. Decrease $c'_d$, if
necessary, so that it is less than $1/(2\alpha)$. Maximizing in $I>0$ proves
\eqref{gm:eq:strongcell} for all sufficiently large $m$.
\end{proof}

\begin{lemma}\label{gm:lem:weightedloss}
Fix $0<\xi\le1/2$. Let $H,r$ be an end subgraph of order $b$, with $\deg_H(r)\ge d-1$ and
every other degree at least $d$. Let $0\le t b^2\le1/2$ and put $\bu=(A_H-tI)^{-1}\1$. Add a
zero vertex adjacent only to $r$ and order its value first, followed by the components of
$\bu$ in nondecreasing order. Suppose that among the first $\lfloor\xi b\rfloor$ difference
indices, the consecutive cuts of size one are separated by at least $D+1$ positions. Then,
for constants $c_{d,\xi}>0$ and $C_d$,
\begin{equation}\label{gm:eq:weightedloss}
 T_H(t)\le \frac{b^3}{3\alpha}-c_{d,\xi}b^3+C_db^2+4tb^5.
\end{equation}
\end{lemma}
\begin{proof}
\Cref{gm:lem:response} makes $A_H-tI$ positive definite.
\Cref{lem:positive} therefore gives $\bu>0$.
Off $r$, the equation $L_Hu=\1+tu>0$ excludes a minimum, so
$u_r$ is the smallest component. Write $a_i$ for the ordered
consecutive differences, including the difference from the added zero vertex. All vertices
except this added vertex have degree at least $d$ in the augmented graph.
The degree counts between cuts still apply, since the counted vertices
exclude the added vertex. The last $d$ cuts are not small, by the
suffix version of \eqref{eq:cutcount}. Hence the selected intervals
are disjoint and fit within the $b$ difference indices, including the
interval beginning at the first cut, which has size one. Put $w_i=b+1-i$. Then
\[
 \sum_{v\in H}u_v=\sum_{i=1}^b w_i a_i.
\]
On a selected interval beginning at $a$, replace $w_i$ by $I_a=w_a$; on an uncovered index
retain $w_i$. As $w_i\le I_a$ and $a_i\ge0$, this can only increase the linear expression.
The sum of the constants used in completing the squares is
\[
 B_0=\frac1\alpha\left(\sum_{B}|B|I_B^2+\sum_{i\in U}w_i^2\right)
 \le\frac1\alpha\sum_{i=1}^b w_i^2+C_db^2
 =\frac{b^3}{3\alpha}+O_d(b^2).
\]
Indeed, on every interval $0\le I_B-w_i\le D$, so replacing the squares costs at most $C_db^2$ in total.

Put $\mathcal R=B_0-2\sum_vu_v+\bu^\top A_H\bu$ and define
\[
 p_i=\begin{cases}I_B/\alpha,&i\in B,\\ w_i/\alpha,&i\in U.\end{cases}
\]
The square completions in \Cref{gm:lem:fixedloss}, with
$I_B$ or $w_i$ in place of $I$, apply to this remainder. The
replacement of the weights adds only
$2(I_B-w_i)a_i\ge0$. Thus the endpoint errors, uncovered errors,
nonbridge penalties and adjacent products required by
\Cref{lem:adjacent} are all controlled by $C_d\mathcal R$.
At successive indices $p$ changes by at most $D/\alpha$, and it
is constant within each selected interval. Hence that lemma gives
\[
 \sum_{i\in V}p_i^2\le C_d\mathcal R+C_db.
\]
Put $L=\lfloor\xi b\rfloor$. On the first $L$ indices,
$p_i\ge w_i/\alpha\ge b/(2\alpha)$.
At most $L/(D+1)+O_d(1)$ bridge intervals meet this prefix, by
the assumed separation, and each covers at most $D$ indices.
Therefore
\[
 |V\cap\{1,\ldots,L\}|\ge L/(D+1)-C_d,
 \qquad
 \sum_{i\in V}p_i^2\ge c_{d,\xi}b^3-C_db^2.
\]
It follows that $\mathcal R\ge c_{d,\xi}b^3-C_db^2$, and thus
\[
 2\sum_vu_v-\bu^\top A_H\bu
 \le b^3/(3\alpha)-c_{d,\xi}b^3+C_db^2.
\]
Finally, the defining equation implies
\[
 T_H(t)=2\sum_v u_v-\bu^\top A_H\bu+t\|\bu\|^2.
\]
Use $\|\bu\|^2\le4b^5$ from \Cref{gm:lem:response}. Boundedly many small values of $b$ are absorbed by $C_db^2$.
\end{proof}

\subsection{Comparison subgraphs in the two classes}
We build replacement pieces that satisfy the required degree conditions
in each of the two classes.

\begin{lemma}\label{gm:lem:comparisons}
In the minimum-degree class, for every sufficiently large integer $m$ there is a graph $J_m$ with two specified vertices, satisfying
the degree conditions of \Cref{gm:thm:terminal}, with
\[
 R_{J_m}\ge m/\alpha-C_d.
\]
For every sufficiently large integer $b$ there is an admissible end subgraph $C_b$ with
\[
 T_{C_b}(0)\ge b^3/(3\alpha)-C_db^2.
\]
\end{lemma}
\begin{proof}
Write $m=kD+r$, $0\le r<D$. Use $k-1$ copies of $L_d$ followed by $K_{D+r}$, joined by
bridges, with distinct specified vertices at the two outer ends. Its value of $R_H$ is
$(k-1)D/\alpha+2/(D+r)$, which proves the first claim.
Use the same chain as an end subgraph, rooted at the first missing-edge endpoint, with a
complete graph at its far end. Add the external zero vertex and solve $A_{C_b}u=\1$. If a
copy of $L_d$ has $I$ vertices to the right of its incoming bridge, including
its own vertices and all later blocks, the difference across its incoming bridge is $I$, and its two
nonzero internal differences
are $(I-1)/\alpha$ and $(I-d)/\alpha$. Its contribution to $u^TA_{C_b}u=T_{C_b}(0)$ is
exactly
\[
 I^2+\frac{(I-1)^2+(I-d)^2}{\alpha}
 =\frac D\alpha I^2-\frac{2D}\alpha I+\frac{1+d^2}\alpha.
\]
Here $I=b,b-D,\ldots$, down to a value bounded in terms of $d$. Summing the squares gives
$b^3/(3\alpha)+O_d(b^2)$; the end graph has bounded order and does not affect the
estimate.
\end{proof}

\begin{lemma}\label{or:lem:comparisons}
Suppose $d$ is odd. For every sufficiently large even $m$, there is a connected 
subgraph $J_m$ of order $m$ with two specified vertices of degree $d-1$,
all other degrees $d$, and
\[
 R_{J_m}\ge m/\alpha-C_d.
\]
For every sufficiently large odd $b$, there is a connected  end
subgraph $C_b$, with its specified vertex of degree $d-1$ and all other
degrees $d$, such that
\[
 T_{C_b}(0)\ge b^3/(3\alpha)-C_db^2.
\]
\end{lemma}
\begin{proof}
Let $Q_N$ be the connected graph constructed in \Cref{sec:upper},
where $N$ is even and $D\le N\le2d$. Its only deficient vertices
are $0,1$, each of degree $d-1$.
For every sufficiently large even $m$, choose $N$ in this range with
$m\equiv N\pmod D$, and put $m=kD+N$. Join $k$ copies of $L_d$ and
$Q_N$ by bridges. The specified vertices are the two unused deficient
vertices at the ends. All other deficient vertices receive one bridge.
This gives the exact stated degrees. Minimizing with prescribed
endpoint values successively along the chain gives
\[
 R_{J_m}=k(1+2/\alpha)+R_{Q_N}=kD/\alpha+R_{Q_N}
        \ge m/\alpha-C_d.
\]
Only the finitely many bounded graphs $Q_N$ enter the last constant.

Let $E_d$ be the end graph of order $d+2$ constructed in
\Cref{sec:upper}. Given a sufficiently large odd $b$, write
\[
 b=kD+N+(d+2)
\]
with $N$ as above. Join $k$ copies of $L_d$, then $Q_N$, then $E_d$.
Specify the first unused deficient vertex. Again every other vertex
has degree exactly $d$. The last two subgraphs have bounded total order.

The calculation in \Cref{gm:lem:comparisons} applies to each
copy of $L_d$ and its incoming bridge: summing $A_{C_b}\bu=\1$
on the suffix gives the same sum of differences $I$ across the incoming bridge,
irrespective of the bounded subgraph at the end. Thus these copies contribute
$b^3/(3\alpha)+O_d(b^2)$, and the remaining contribution is
nonnegative. This proves the required bound for $T_{C_b}(0)$.
\end{proof}

The parities in this lemma are exactly the parities of portions of an
odd-regular graph. An interior portion with two outside edges has even
order, since $d|H|-2=2e(H)$; an end subgraph with one outside edge has
odd order, since $d|H|-1=2e(H)$. Thus no order adjustment is required
when these subgraphs are substituted.

For the remainder of the proof, $J_m$ and $C_b$ denote the
comparison subgraphs in the chosen class. Substitution is therefore
always at the same order. In the regular class the two specified vertices
of $J_m$, or the one specified vertex of $C_b$, have degree exactly $d-1$;
reattaching the boundary bridges restores every degree to $d$.

\subsection{The order of every exceptional interval is bounded}
We show that every exceptional interval has bounded size: a larger one
could be replaced to lower algebraic connectivity.

By \Cref{lem:comparisonupper} and exact minimality, a minimizing
graph in the chosen class satisfies
\[
 \mu(G)\le\frac{(d-1)\pi^2}{n^2}+\frac{A_d}{n^3}.
\]
For sufficiently large $n$, apply \Cref{thm:stability} with
$\varepsilon=A_d/n$. We obtain a nondecreasing Fiedler order in
which the gaps between consecutive cuts of size one that induce
$L_d$ cover $n-O_d(n^{2/3})$ vertices. In particular,
$\mu(G)=O_d(n^{-2})$. This deduction uses only the comparison upper
bound and does not use \Cref{thm:minimum}. The proof of \Cref{sec:stability} selects these blocks
from the list of all ordered bridges; this ordered form, rather than just the number of
covered vertices, is what we use here. Only the consequence that the uncovered number is
$o_d(n)$ is needed below.

In this ordering, call a gap between consecutive cuts of size one a standard gap if it induces
$L_d$ with its required attachments. All other gaps, as well as the initial and final
portions, are exceptional. A maximal consecutive union of exceptional gaps is an exceptional interval;
include the initial or final portion in the adjacent exceptional interval when
appropriate. Every such portion is connected, by the observation following
\Cref{gm:lem:bridgeextrema}. An interior portion has exactly
two outside edges, and an end subgraph has exactly one.

\begin{lemma}\label{gm:lem:smallestgap}
In an exact minimizer, every gap of order $D$ between consecutive
ordered bridges is standard. Every exceptional interior gap has at least $D+1$
vertices; in the odd-regular class its order is even and at least $D+2$.
\end{lemma}
\begin{proof}
The attachments are distinct by \Cref{gm:lem:attachments}.
A gap of order $D$ has exactly two outgoing edges. Degree counting
shows that it induces $K_D$ or $K_D-e$. In the latter case the
deficient vertices must receive the two bridges. In the regular
class $2e(H)+2=dD$ already excludes $K_D$. In the minimum-degree
class, both attachment vertices of $K_D$ have degree $d+1$, so
their joining edge may be deleted. Their Fiedler components are distinct: by \Cref{gm:lem:bridgeextrema}
they are the minimum and maximum on the gap, and equality would make every component on
the gap constant. The equations at an interior vertex and an attachment would then
contradict the strict difference on the incoming bridge. Deleting the edge strictly lowers
the Rayleigh quotient. Finally, two cuts of size one are separated by at least $D$
positions, by the cut-counting inequality. In odd regular degree,
$d|H|-2=2e(H)$ makes the order even, proving the last assertion.
\end{proof}

\begin{theorem}\label{gm:thm:boundedintervals}
In each chosen class there are constants $B_d,N_d$ such that every
exceptional interval of every exact minimizer of order $n\ge N_d$
contains at most $B_d$ vertices.
\end{theorem}
\begin{proof}
Otherwise there is a sequence of exact minimizers with an exceptional interval of order
$m\to\infty$. \Cref{thm:stability} gives $m=o(n)$, and in fact the total number of
vertices in all exceptional intervals is $o(n)$. Put $\mu=\mu(G)$ and let $\x$ be the
corresponding mean-zero unit Fiedler vector. By \Cref{lem:smallentries},
$\max_v|x_v|\le C_dn^{-1/2}$, and at least $c_dn$ components have
each strict sign. These facts will be used in both replacements.

First consider an interior exceptional interval $H$, with $p$ vertices to its left and $q$
to its right. By \Cref{gm:lem:bridgeextrema}, its attachment vertices have the minimum and maximum Fiedler
components on $H$, denoted by $a,b$, respectively. They satisfy $a<b$: equality would make the interval constant; an
interior equation would make this constant zero, and an attachment equation would
contradict the strict difference on its outside bridge. Let $\Delta=b-a$,
$M=\max\{|a|,|b|\}$, and $t=\min\{p,q\}$. Lemmas~\ref{gm:lem:attachments}, \ref{gm:lem:fixedloss}
and~\ref{gm:lem:smallestgap} give
\[
 E_H:=\E_H(\x)\ge(\alpha+\eta_d)\Delta^2/m.
\]
Replace $H$ by the same-order subgraph $J_m$ in the chosen class and
choose the components minimizing its quadratic form with the same
endpoint values. In the regular class $m$ is even, the two
attachments in $H$ have internal degree $d-1$, and all other
internal degrees are $d$. The replacement has exactly these degree
deficiencies. Reattaching the bridges gives a connected 
$d$-regular graph. In the minimum-degree class the same construction
preserves the required degree lower bounds. For large $m$, its quadratic form is at most
$(1-\chi_d)E_H$, with $\chi_d>0$ fixed. Every new component lies in $[a,b]$. The loss of
centered squared norm is therefore at most
\[
 2mM\Delta+m^2\Delta^2/n\le4mM\Delta.
\]
For a cut inside $H$, the sum of the differences on its crossing
edges is $-\mu\sum_{v\le i}x_v$. If $M=|a|$ and $a\le0$, it is at least $\mu(p-m)M$; if
$M=b\ge0$, use the suffix to obtain $\mu(q-m)M$. The two boundary bridges do not cross an interior cut with a nonzero
consecutive difference. Their attachment values are the minimum and maximum
on $H$; cuts through a tie at either endpoint have zero consecutive difference. Hence
\[
 E_H=\sum_{i\text{ inside }H}a_i\left(-\mu\sum_{v\le i}x_v\right)
 \ge\mu(t-m)M\Delta.
\]
If $t\ge K_dm$, with $K_d>1+4/\chi_d$, the decrease in the quadratic form exceeds $\mu$ times the possible
norm loss. This strictly improves the Rayleigh quotient, a contradiction.

It remains $t<K_dm$, or that the exceptional interval meets an end. In the former case
include the entire nearer end up to the far bridge of $H$, obtaining an end subgraph of order
$b_0\le(K_d+1)m=o(n)$. In the latter case take that end subgraph directly. Reverse the order
when needed, so that its first $m$ vertices from its attachment belong to exceptional gaps. The sign
count in \Cref{lem:smallentries} shows that its outside attachment has a nonzero
component and that every component on this end subgraph has the same sign. Writing the
outside value as $a_0$, its vector is
\[
 \x_H=a_0\bigl(\1+\mu(A_H-\mu I)^{-1}\1\bigr).
\]
Since $a_0\ne0$, this formula is a strictly monotone affine
function of the components of $(A_H-\mu I)^{-1}\1$. Starting at
the attachment and moving towards the end therefore gives exactly
the order in \Cref{gm:lem:weightedloss}, after reversing the
original order when necessary. Ties are kept in that inherited
order. Thus its hypothesis about cuts is a property of the chosen
Fiedler ordering, not an assumption about a new ordering. The first $\lfloor m/2\rfloor$
differences have successive cuts of size one separated by at least $D+1$, by
\Cref{gm:lem:smallestgap}. Choose the fixed value $\xi=1/(4(K_d+1))$. Then
$\lfloor\xi b_0\rfloor\le\lfloor m/2\rfloor$, so the same separation
holds on the prefix required by \Cref{gm:lem:weightedloss}.

The end replacement $C_{b_0}$ is admissible in the chosen class.
In the regular case $b_0$ is odd and its single deficient vertex
receives the outside bridge, restoring degree $d$; in either case
the graph stays connected and simple. Since $\mu b_0^2=o(1)$,
\Cref{gm:lem:weightedloss} and the corresponding comparison
lemma imply, for sufficiently large $m$,
\[
 T_H(\mu)<T_{C_{b_0}}(0)\le T_{C_{b_0}}(\mu).
\]
The two subgraphs have the same order. The difference of their values of $F$ is positive, and
\Cref{gm:lem:replacecap} therefore gives a strict improvement, another contradiction.
No exceptional interval can consequently have unbounded order along such a sequence. This
is equivalent to the assertion of the theorem.
\end{proof}

\subsection{Location of the bounded exceptional subgraphs}
We show that an exceptional piece cannot stay deep inside the chain,
so all such pieces lie near the ends.

For this comparison it is convenient to put a point at the middle of each of the two
boundary bridges of an interior subgraph. Each half-edge has coefficient $2$ in the quadratic
form and the inserted middle points do not contribute to the squared norm. Minimizing over its value
replaces the two half-edges by the original unit edge. The comparison below keeps the
boundary middle-point values fixed; the quadratic form of the resulting ordinary graph is no
larger than the quadratic form before minimizing those points.

\begin{lemma}\label{gm:lem:commutator}
Let $H$ be a connected subgraph of order $m$ with distinct specified
vertices satisfying the degree conditions of \Cref{gm:thm:terminal},
and let $C=L_d$. Include one half
of each boundary bridge in each subgraph, and put
\[
 r_H=1+R_H,\qquad r_C=D/\alpha,\qquad
 \delta_H=mr_C-Dr_H.
\]
Let $z_v$ be the components minimizing the quadratic form in the augmented subgraph, with
boundary values $0,r_H$. Its boundary difference and quadratic form both equal $r_H$. Put
\[
 S_H=\sum_{v\in H}z_v-\frac{mr_H}{2}.
\]
For boundary values $a,b$ and sufficiently small $t\ge0$, let
$q_{HC}(t;a,b)$ be the minimum of the quadratic form minus $t$ times
the squared norm on the two consecutive subgraphs. Define $q_{CH}$ in the same way. With
$r=r_H+r_C$, $c=(a+b)/2$ and $J=(b-a)/r$, one has
\begin{equation}\label{gm:eq:commutator}
 q_{HC}(t;a,b)-q_{CH}(t;a,b)
 =2t\delta_H cJ+2tr_CS_HJ^2+O_{H,d}\bigl(t^2(a^2+b^2)\bigr).
\end{equation}
For a fixed bound on $m$, the remainder and its validity interval can be chosen uniformly
over all permitted subgraphs and orientations; this follows from finiteness of the family. If $H$ is not a chain of $L_d$ subgraphs, then $\delta_H>0$.
\end{lemma}
\begin{proof}
The last assertion is \Cref{gm:thm:terminal}. To obtain the expansion, let $B$ be the positive definite
matrix on the free graph vertices after fixing the two outside
middle-point values and eliminating the internal middle point.
The middle points have no weight in the squared norm.
Apply \eqref{eq:boundary-expansion} with $\omega=0$ and with
$\mathbf f_0$ the minimizing vector at $t=0$. For
$0\le t\le\lambda_{\min}(B)/2$, the inverse has norm at most
$2/\lambda_{\min}(B)$, while $\|\mathbf f_0\|^2=O_{H,d}(a^2+b^2)$.
This proves the stated remainder, uniformly for a finite family.
For both orders, the minimum quadratic form is $q(0)=(b-a)^2/r$. The corresponding minimizing
coordinates in $HC$ are $z_v$ on $H$ and $r_H+z_v$ on $C$; in $CH$ they are $z_v$ on $C$
and $r_C+z_v$ on $H$. Symmetry gives $\sum_Cz_v=Dr_C/2$. The difference of the first
coordinate sums is $-\delta_H$, and that of the second coordinate sums is
\[
 -r\delta_H-2r_CS_H.
\]
As $\mathbf f_0=a\1+J\z$, the difference of the squared norms is $-2\delta_HcJ-2r_CS_HJ^2$.
Substitution proves \eqref{gm:eq:commutator}.
\end{proof}

The minimizing equation in \eqref{eq:boundary-expansion} also gives
the following error estimate for bounded subgraphs.
Let $\mathbf f_0$ minimize the quadratic form with fixed boundary values,
and let $\mathbf f$ be the restriction of a Fiedler vector. If $B$ is the
matrix on the unfixed graph vertices, then
\begin{equation}\label{gm:eq:harmonicerror}
 B(\mathbf f-\mathbf f_0)=\mu \mathbf f.
\end{equation}
The path estimate in \Cref{gm:lem:response} bounds $\|B^{-1}\|$
in terms of the number of vertices. Hence, for bounded subgraphs,
each component of $\mathbf f-\mathbf f_0$ is $O(\mu M)$, with a constant depending
only on $d$ and the bound on the order. Here $M$ bounds
the original components and the boundary values. The same bound
holds for the change of the sum of differences at either boundary.
The inserted middle points are eliminated before applying this
identity; they have no term in the squared norm.

\begin{lemma}\label{gm:lem:defectmoves}
Fix a bound $B$ and a constant $C_0>0$. There are constants
$T_{d,B,C_0},N_{d,B,C_0}$ with the following property. Let $G$ have
minimum degree at least $d$, order $n\ge N_{d,B,C_0}$, and
$\mu(G)\le C_0n^{-2}$. Suppose its ordered bridges
meet the extremal components on their two sides. Let an interior
exceptional interval $H$ have distinct attachments, at most $B$ vertices,
and a standard gap immediately on both sides. If neither interchange of
$H$ with a neighbouring standard gap decreases algebraic connectivity,
then at least one outside part of $H$ has at most $T_{d,B,C_0}$ vertices.
\end{lemma}
\begin{proof}
All subgraphs, their endpoint choices and their orientations now belong to a finite family.
First, an exceptional interval cannot itself be a chain of $L_d$ with the indicated
endpoint vertices, for sufficiently small $\mu$. Indeed, the twin vertices in each middle
clique have equal components, by subtracting their eigen-equations. If the two endpoint
values are both nonnegative or both nonpositive, all interior values have the same weak
sign by bridge extremality. Summing the eigen-equations over successive
groups shows that the sums of differences across the natural interfaces form a monotone
sequence. Since the sums at both boundary bridges are positive, every intermediate
sum is positive. Applying the same summation to cuts between a singleton
and its adjacent middle clique shows that all successive distinct
vertex groups have strictly increasing values. If the endpoint values have opposite strict signs, their largest absolute value
is at most their difference. On this bounded chain, \eqref{gm:eq:harmonicerror} bounds the
difference from the minimizing vector with those endpoints by $C\mu$ times the endpoint
difference. The differences in the minimizing vector between successive distinct groups
are fixed positive multiples of the endpoint difference. They therefore remain positive
for small $\mu$. In both cases the natural bridge cuts are ordered and all intervening
gaps are standard, contrary to the definition of $H$.

Thus \Cref{gm:lem:commutator} gives a positive lower bound for $\delta_H$ on the
finite family of possible exceptional intervals. All quantities implicit below are uniform
in that family. Let $C,H,C$ be the three consecutive subgraphs. Write $a_H,b_H$ for the two
boundary middle-point values of $H$, and put
\[
 c_H=(a_H+b_H)/2,\qquad I=(b_H-a_H)/r_H>0.
\]
Write $M$ for the largest absolute value on the three subgraphs and their boundary points.
Equation~\eqref{gm:eq:harmonicerror} shows that on each bounded subgraph the Fiedler vector
differs by $O(\mu M)$ from the vector minimizing its quadratic form with the same boundary
values. Thus each boundary sum differs by $O(\mu M)$ from the endpoint
difference divided by $r_H$ or $r_C$, as appropriate. At a shared middle point
the two sums agree, because the value there is the average of the
two ends of the original bridge. It follows that the corresponding
parameters for the neighbouring $C$ subgraphs are $I+O(\mu M)$ and that
\begin{align*}
 c_L&=c_H-r_CI/2+O(\mu M),&J_L&=I+O(\mu M),\\
 c_R&=c_H+r_CI/2+O(\mu M),&J_R&=I+O(\mu M),
\end{align*}
where $c_L,J_L$ refer to the union $CH$ and $c_R,J_R$ to $HC$. Also
\[
 M\le C(|c_H|+I)
\]
for sufficiently small $\mu$. One may see this first with an additional $C\mu M$ on the right and absorb that term.

Keep the outside values fixed and optimize after either interchange. The two changes in
the quadratic form minus $\mu$ times squared norm, before centering, have sum
\begin{equation}\label{gm:eq:swapsum}
 -2\mu\delta_Hr_CI^2+O(\mu^2M^2).
\end{equation}
Indeed, apply \eqref{gm:eq:commutator} to $CH\to HC$ and $HC\to CH$. The terms involving
$c_H$ and $S_H$ cancel, leaving the displayed negative term. The changed real components
differ from their old ones by $O(I+\mu M)$, so centering adds at most
\begin{equation}\label{gm:eq:centerswap}
 C\mu(I+\mu M)^2/n
\end{equation}
to either change. The quadratic form of the ordinary graph is no larger than the quadratic form calculated with
the fixed middle-point values.

Choose a sufficiently large constant $K$. If $|c_H|\le KI$, then $M=O_K(I)$, and \eqref{gm:eq:swapsum},
including both centering corrections, is negative for sufficiently large $n$. At least one
interchange strictly improves the Rayleigh quotient.

Suppose $c_H>KI$; the negative case follows by reversal and a sign change. Let $t$ be the
smaller number of vertices outside the three subgraphs. For $K$ sufficiently large, every
component to the right is at least $c_H/2$. By bridge extremality, every
component beyond the right boundary is at least the last component of the three subgraphs.
Summing the eigen-equations over that suffix therefore shows that the difference on the
boundary bridge is at least $\mu t c_H/2$. Its difference from $I$ is $O(\mu M)$, and
$M=O(c_H)$, so
\[
 \frac{\mu M}{I}\le\frac C t
\]
when $t$ exceeds a fixed constant. For the right interchange, \eqref{gm:eq:commutator} has
leading term $-2\mu\delta_Hc_RJ_R$. Choose $K$ large enough to absorb $2\mu
r_C|S_H|J_R^2$, and then $t$ large enough to absorb the error $O(\mu^2M^2)$. The resulting
change is at most $-c\mu c_HI<0$. The centering correction \eqref{gm:eq:centerswap} is
smaller for sufficiently large $n$. This again contradicts the assumed nondecrease.

Thus $t$ is bounded. Adding the bounded sizes of the two neighbouring standard gaps gives the
asserted bound on an outside part of $H$.
\end{proof}

\begin{theorem}\label{gm:thm:boundedends}
In each chosen class, every sufficiently large exact minimizer consists of an
uninterrupted chain of $L_d$ blocks joined by bridges and two connected end subgraphs
containing $O_d(1)$ vertices in total.
\end{theorem}
\begin{proof}
By \Cref{gm:thm:boundedintervals}, every exceptional interval, including the two
end intervals, has bounded order. Every interior exceptional interval has a standard gap
immediately on both sides. Interchanging the disjoint subgraphs and
reconnecting their specified attachment vertices by bridges preserves connectedness,
simplicity and every degree: each attachment loses one outside
edge and gains one, and all other degrees are unchanged.
Thus both interchanges are admissible even in the regular class.
By \Cref{lem:comparisonupper}, we may fix $C_0=C_0(d)$ such that
$\mu(G)\le C_0n^{-2}$ for every minimizer under consideration.
Exact minimality, together with
Lemmas~\ref{gm:lem:bridgeextrema} and~\ref{gm:lem:attachments}, therefore
allows \Cref{gm:lem:defectmoves} to place each such interval within
a bounded number of vertices of an end. If the two bounds are denoted by $B$ and $T$, every
exceptional
vertex is among the first $T+B$ or the last $T+B$ positions.
Extending these two portions to the next boundary bridges adds at
most $O_d(1)$ vertices. Thus all exceptional vertices lie in two
bounded end subgraphs. Between them every bridge gap is standard, giving the stated chain.
Increasing the order threshold, if necessary, ensures that this middle chain contains at
least two copies.
\end{proof}

\subsection{Cut vertices in the bounded end subgraphs}
We rule out branching at a cut vertex by showing that such a graph
could be rearranged to lower algebraic connectivity.

The end subgraphs given by \Cref{gm:thm:boundedends} are
initial and final segments of the nondecreasing order. By
\Cref{lem:smallentries}, every component on either portion
has the same strict sign. Change the sign of the whole vector
when needed, so that the components on the subgraph under consideration
are positive.
The following rearrangement applies in both classes.

\begin{lemma}\label{gm:lem:graft}
Let $\x$ be a mean-zero unit Fiedler vector of a graph $G$ in the chosen class. Suppose
that $G-v$ has at least three components, and that two of them, $A,B$, have
$x_w>\max\{x_v,0\}$ at every vertex. Then $G$ is not a minimizer.
\end{lemma}
\begin{proof}
Put $F_A=\sum_{w\in A}x_w$, $F_B=\sum_{w\in B}x_w$, $p=|N(v)\cap A|$, and $q=|N(v)\cap
B|$. Relabel the two components so that $F_A/p\le F_B/q$. Choose $u\in N(v)\cap A$ with
the smallest component and put $\eta=x_u-x_v>0$. Summing the eigen-equations on $A$ gives
\begin{equation}\label{gm:eq:graft1}
 p\eta\le\sum_{w\in N(v)\cap A}(x_w-x_v)=\mu F_A,
 \qquad q\eta\le\mu F_B.
\end{equation}
Let $s=\deg(v)-d$. First suppose $s\ge q$. Move every edge from $v$ to $B$ so that its
endpoint $v$ is replaced by $u$. The new graph is simple, connected and has minimum degree
at least $d$. Increase all components on $B$ by $\eta$. Its quadratic form is unchanged, whereas
its centered squared norm increases by
\[
 2\eta F_B+|B|(1-|B|/n)\eta^2>0.
\]
The Rayleigh quotient is therefore smaller.

Suppose instead that $s<q$, and put $c=q-s$. Let
\[
 Z=\{w\in N(u)\cap A:vw\notin E(G)\},\qquad k=|Z|.
\]
There are at most $p-1$ common neighbours of $u,v$ in $A$, so $k\ge d-p$. At least one
edge at $v$ leads to a third component. Thus $d-p\ge c+1$, and $k\ge c+1\ge2$.
By the choice of $u$, every common neighbour in $A$ has component at least $x_u$. The eigen-equation at $u$ gives
\[
 \sum_{w\in Z}(x_w-x_u)\le\eta-\mu x_u.
\]
Take the $c$ smallest components in $Z$, forming $W$. Then
\begin{equation}\label{gm:eq:graft2}
 \sum_{w\in W}(x_w-x_u)\le\frac c k(\eta-\mu x_u).
\end{equation}
Move $v$--$B$ to $u$--$B$, and move $uw$ to $vw$ for all $w\in W$. The degree of $v$
becomes $d$, the degree of $u$ increases by $s$, and all other
degrees are unchanged. In the regular class $s=0$, so the compensated
case applies and every degree is preserved. All added edges are new. The edge $uv$ is retained; any part separated from $u$ by the deleted
edges is reconnected to $v$ by a new edge. Hence the graph remains connected.
Again add $\eta$ to every component on $B$. The moved boundary edges retain their
differences. The change of quadratic form is
\[
 c\eta^2+2\eta\sum_{w\in W}(x_w-x_u),
\]
while the change of variance is $2\eta F_B+|B|(1-|B|/n)\eta^2$. By
\eqref{gm:eq:graft1}--\eqref{gm:eq:graft2}, the change of the quadratic form minus $\mu$ times the centered squared norm is at most
\[
 -c\eta^2\left(1-\frac2k\right)
 -\frac{2c\mu x_u\eta}{k}
 -\mu|B|(1-|B|/n)\eta^2<0.
\]
The new centered vector therefore has smaller Rayleigh quotient. Its variance is nonzero,
since a constant vector has zero quadratic form and cannot make this last expression negative.
\end{proof}

Let $B$ bound the end orders. If $U$ is an outward component of
$G-v$, its principal Laplacian $L_U$, retaining the full degrees,
has smallest eigenvalue at least $B^{-2}$. Indeed, extend a vector
by zero at $v$ and use paths of at most $B$ edges from its vertices
to $v$. Cauchy--Schwarz and summation give
$\|\z\|^2\le B^2\z^\top L_U\z$.
For large $n$, $\mu<B^{-2}$, and the eigen-equations together
with \Cref{lem:positive} give, whenever $x_v>0$,
\begin{equation}\label{gm:eq:outward}
 \x_U-x_v\1=\mu x_v(L_U-\mu I)^{-1}\1>\0.
\end{equation}
Here $L_U\1$ is the vector indicating the neighbours of $v$ in $U$.

\begin{lemma}\label{gm:lem:cutvertices}
No vertex of a sufficiently large minimizer in the chosen class
separates it into three or more components.
\end{lemma}
\begin{proof}
A vertex in the middle $L_d$ chain has at most two components after
deletion. Consider a vertex $v$ in an end subgraph, and choose the sign of $\x$
so that its components on that subgraph are positive.
If $G-v$ has at least three components, choose two, $A,C$, not
containing the middle chain. Their orders are bounded, and
\eqref{gm:eq:outward} makes every component on them greater than
$x_v>0$. \Cref{gm:lem:graft} gives a smaller Rayleigh
quotient in the chosen class, a contradiction.
\end{proof}

\subsection{The block-tree is a path}
We rule out the remaining type of branching and conclude that the
block-tree is a path.

It remains to exclude branching at a block-node with two different
outward cut vertices. The following switches preserve every degree,
so the same argument works in both classes.

\begin{lemma}\label{gm:lem:maxedge}
Let $G$ be a sufficiently large exact minimizer in the chosen class,
and let $A$ be a component of $G-a$ that does not contain the middle
chain. Suppose $A$ lies in a bounded end subgraph on which the
Fiedler components are positive.
Choose $w\in A$ with largest component
on $A$. Then $G-w$ is connected. There is a neighbour $z\in A$ of $w$,
and deleting $wz$ leaves $G[A\cup\{a\}]$ connected.
\end{lemma}
\begin{proof}
If $G-w$ were disconnected, a component $U$ not containing $a$ would
lie in $A\setminus\{w\}$. Its order is bounded, and
\eqref{gm:eq:outward}, with $w$ in place of $v$, would make all its
components greater than $x_w$. This contradicts the choice of $w$.
Thus $G-w$ is connected. It follows also that
$G[A\cup\{a\}]-w$ is connected: a path leaving this subgraph can do
so only at $a$, and the excursion can be deleted.
Only $a$ can be a neighbour of $w$ outside $A$, while $\deg(w)\ge d\ge3$.
Choose a neighbour $z\in A$. After deleting $wz$, the remaining graph
on $(A\cup\{a\})\setminus\{w\}$ is still connected and $w$ still
has another neighbour there. This proves the last assertion.
\end{proof}

\begin{lemma}\label{gm:lem:levelswitch}
Under the hypotheses of \Cref{gm:lem:maxedge}, no edge $uv$, with $x_u<x_v$ and both endpoints outside
$A\cup\{a\}$, satisfies
\[
 x_u\le x_w\le x_v.
\]
\end{lemma}
\begin{proof}
Choose $z$ as in that lemma, so $x_z\le x_w$. Replace $wz,uv$ by
$wv,zu$. These edges are absent because no vertex of $A$ has a
neighbour outside $A$ other than possibly $a$. They preserve every
degree.

To verify connectedness, put $R=G-A$, which is connected.
After deleting $wz$, the part $G[A\cup\{a\}]$ is connected. If
$R-uv$ is connected, the whole graph remains connected before adding
new edges. Otherwise $R-uv$ has two components, containing $u$ and
$v$ respectively; the new edges join both of them to the connected
part containing $w,z$. Hence the new graph is connected in either case.

On the unchanged Fiedler vector the change in the quadratic form is
\begin{equation}\label{gm:eq:levelswitch}
 2(x_w-x_u)(x_z-x_v)\le0.
\end{equation}
If it is strict we have a contradiction. If $x_w=x_u$, then
$x_v>x_w\ge x_z$ and the eigen-equation at $w$ changes by $x_z-x_v\ne0$.
The other equality possibility is $x_z=x_v$, which forces
$x_z=x_w=x_v>x_u$; then the equation at $z$ changes by $x_w-x_u\ne0$.
These cover all equality cases, because $x_u<x_v$. Thus equality in
the Rayleigh principle is impossible as well.
\end{proof}

\begin{lemma}\label{gm:lem:blockbranch}
In a bounded end subgraph of a sufficiently large minimizer,
no block has two distinct outward cut vertices.
\end{lemma}
\begin{proof}
By changing the sign of the Fiedler vector if necessary, assume that
all its components on the end subgraph are positive. Root the
block-tree of this subgraph at its attachment to the middle chain. Suppose a block $B$ has outward cut vertices $a,b$, with the
outward components $A,C$, respectively. By
\Cref{gm:lem:cutvertices}, these are the only components beyond
those vertices away from $B$. They are bounded, and
\eqref{gm:eq:outward} gives
\[
 \x_A>x_a\1,\qquad \x_C>x_b\1.
\]
Write $M_A=\max_Ax$, $M_C=\max_Cx$, relabel so that $M_A\le M_C$,
and choose $w\in A$ with $x_w=M_A$.

If $x_b\le M_A$, a path from $b$ to a vertex of $C$ with value $M_C$
contains a nonconstant edge with $M_A$ between its endpoint values.
This includes the equality case $x_b=M_A$, since then $M_C>x_b$.
Both endpoints of that edge lie outside $A\cup\{a\}$, contradicting
\Cref{gm:lem:levelswitch}. Therefore $x_b>M_A$.

The graph $B-a$ is connected. If it contained a vertex of value at
most $M_A$, a path there to $b$ would give the same contradiction.
Thus every vertex of $B-a$ has value greater than $M_A>x_a$.
Every neighbour of $a$ outside $A$ is in $B$, since $a$ belongs to
exactly two blocks by \Cref{gm:lem:cutvertices}. All neighbours
of $a$ now have component greater than $x_a>0$, contradicting
\[
 \sum_{v\sim a}(x_a-x_v)=\mu x_a>0.
\]
This excludes the proposed branching.
\end{proof}

\begin{proof}[Proof of \Cref{thm:GMfull}]
\Cref{gm:thm:boundedends} leaves an uninterrupted $L_d$ chain
and two bounded end subgraphs. \Cref{gm:lem:cutvertices}
excludes branching at cut-vertex nodes of the block-tree.
A branching block-node in an end subgraph would have at least two
outward cut vertices and is excluded by
\Cref{gm:lem:blockbranch}. The middle blocks and their bridges
already form a path and have no other exits. Thus the whole
block-tree is a path, including both ends. All blocks other than
the middle $L_d$ blocks and their bridges have $O_d(1)$ vertices
in total, at the two ends. The initial reduction covers every
minimizer in the class $\delta=d$ as well as $\delta\ge d$.

If the middle chain has $k$ copies, then $kD=n-O_d(1)$, and every
path across it has at least $3k-1$ edges.
The diameter bound in \cite{Caccetta,Erdos} and
\cite[Theorems~5.1--5.2]{AbGhDiam} supplies the matching upper bound.
Consequently $\dm(G)=3n/(d+1)+O_d(1)$. This proves all assertions
of the theorem.
\end{proof}

\section{The minimum algebraic connectivity}\label{sec:chainvalue}
In this section, we deduce the numerical formula from the structure of
exact minimizers. We use the chain-to-path comparison from
\cite[Section~3.2]{AbGhDiam}, with the error estimates written explicitly
so that the remainder is $O_d(n^{-3})$, rather than only $o(n^{-2})$.
Put $D=d+1$ and $\alpha=d-1$.

\subsection{The algebraic connectivity of a chain}
We take one Fiedler component from each $L_d$ block and compare the
resulting vector with a vector on a path. The bounded end subgraphs
contribute only a small error to its norm and mean.

\begin{lemma}\label{lem:chainvalue}
Fix $d\ge3$ and an integer $B\ge2$. Let $G$ be obtained from a chain
of $m$ disjoint copies of $L_d$ by joining a connected end subgraph
to each outer singleton by one bridge. Suppose that the two end
subgraphs are disjoint from the chain and have at most $B$ vertices
altogether, and that there are no other edges between these subgraphs
and the chain. If $n=|V(G)|$, then, as $m\to\infty$,
\begin{equation}\label{eq:chainvalue}
 \mu(G)=\frac{(d-1)\pi^2}{n^2}+O_{d,B}(n^{-3}).
\end{equation}
The error is uniform over the choices of the two end subgraphs.
\end{lemma}

\begin{proof}
Write the copies of $L_d$ in their chain order as $C_1,\ldots,C_m$.
In $C_j$, let $s_j$ and $t_j$ be its left and right singleton vertices,
and let $W_j$ be its set of $\alpha$ middle vertices. Thus the bridge
between consecutive copies is $t_js_{j+1}$. If the end subgraphs
have orders $b_-$ and $b_+$, then
\[
 n=Dm+b_-+b_+.
\]

The trial vector in the proof of \Cref{lem:comparisonupper} applies
without change. Indeed, that calculation uses only the $L_d$ chain,
the two attachment bridges, and the bounds on $b_-$ and $b_+$; all
values on an end subgraph are equal, so its internal edges contribute
zero. It gives
\begin{equation}\label{eq:chainupper}
 \mu(G)\le\frac{\alpha\pi^2}{D^2m^2}+O_{d,B}(m^{-4})
          =\frac{\alpha\pi^2}{n^2}+O_{d,B}(n^{-3}).
\end{equation}
In particular, $\mu=\mu(G)=O_{d,B}(n^{-2})$.

Let $\x$ be a mean-zero unit Fiedler vector of $G$. By
\Cref{lem:smallentries},
\begin{equation}\label{eq:chainsize}
 M:=\max_v|x_v|=O_{d,B}(n^{-1/2}).
\end{equation}
We also need a bound on the difference along an edge. Suppose that
$uv\in E(G)$ and $x_u<x_v$, and put
$S=\{w:x_w\le x_u\}$. Summing the eigen-equations over $S$ gives
\[
 -\mu\sum_{w\in S}x_w
 =\sum_{\substack{ab\in E(G)\\a\in S,\ b\notin S}}(x_b-x_a).
\]
Every term on the right is positive, and one of them is $x_v-x_u$.
Cauchy--Schwarz therefore gives
\begin{equation}\label{eq:chainedge}
 |x_v-x_u|\le\mu\sqrt n=O_{d,B}(n^{-3/2})
 \qquad(uv\in E(G)).
\end{equation}
The same bound is immediate when $x_u=x_v$. This is the edge
estimate used in \cite[Section~2]{AbGhDiam}, with its dependence
on $n$ retained.

Put $u_j=x_{s_j}$ and let $\bu=(u_1,\ldots,u_m)^\top$.
Every vertex of $C_j$ is at distance at most two from $s_j$ within
that copy. Hence \eqref{eq:chainedge} gives
\begin{equation}\label{eq:chaincellerror}
 |x_v-u_j|=O_{d,B}(n^{-3/2})
 \qquad(v\in V(C_j)).
\end{equation}
Using \eqref{eq:chainsize}, we obtain, for every $j$,
\begin{align*}
 \left|\sum_{v\in V(C_j)}x_v-Du_j\right|
   &\le\sum_{v\in V(C_j)}|x_v-u_j|
     =O_{d,B}(n^{-3/2}),\\
 \left|\sum_{v\in V(C_j)}x_v^2-Du_j^2\right|
   &\le2M\sum_{v\in V(C_j)}|x_v-u_j|
     =O_{d,B}(n^{-2}).
\end{align*}
The two end subgraphs have at most $B$ vertices. By
\eqref{eq:chainsize}, their total sum is $O_{d,B}(n^{-1/2})$
and their total squared mass is $O_{d,B}(n^{-1})$. Since $\x$
has mean zero and norm one, summing over the $m=O(n)$ copies gives
\begin{align}
 D\sum_{j=1}^m u_j&=O_{d,B}(n^{-1/2}),
     \label{eq:chainmean}\\
 D\sum_{j=1}^m u_j^2&=1+O_{d,B}(n^{-1}).
     \label{eq:chainmass}
\end{align}
If $\bar u=m^{-1}\sum_j u_j$, then \eqref{eq:chainmean} implies
$m\bar u^2=O_{d,B}(n^{-2})$. Consequently,
\begin{equation}\label{eq:chaincenteredmass}
 \|\bu-\bar u\1\|^2
 =\frac1D+O_{d,B}(n^{-1}).
\end{equation}

We next compare the quadratic forms. Let $v_j$ be the average on
$W_j$, and put $w_j=x_{t_j}$. For $j<m$, Cauchy--Schwarz gives
\begin{align*}
 &\sum_{v\in W_j}(x_v-u_j)^2
  +\sum_{v\in W_j}(w_j-x_v)^2
  +(u_{j+1}-w_j)^2\\
 &\qquad\ge
 \alpha(v_j-u_j)^2+\alpha(w_j-v_j)^2+(u_{j+1}-w_j)^2\\
 &\qquad\ge
 \frac{(u_{j+1}-u_j)^2}{\alpha^{-1}+\alpha^{-1}+1}
 =\frac{\alpha}{D}(u_{j+1}-u_j)^2.
\end{align*}
This is precisely the three-cell comparison for the sizes
$(1,d-1,1)$ in \cite[Section~3.2]{AbGhDiam}. No monotonicity
assumption on $u_1,\ldots,u_m$ is needed here. The edge sets used
for different $j<m$ are disjoint, and all omitted terms are
nonnegative. Therefore
\begin{equation}\label{eq:chainpathenergy}
 \mu=\x^\top L(G)\x
 \ge\frac{\alpha}{D}\sum_{j=1}^{m-1}(u_{j+1}-u_j)^2.
\end{equation}
Apply \eqref{eq:pathineq} to $\bu$ and use
\eqref{eq:chaincenteredmass}. Since
$p_m=\pi^2/m^2+O(m^{-4})$, we obtain
\begin{align*}
 \mu
 &\ge\frac{\alpha}{D}p_m\|\bu-\bar u\1\|^2\\
 &\ge\frac{\alpha\pi^2}{D^2m^2}-O_{d,B}(n^{-3})\\
 &=\frac{\alpha\pi^2}{n^2}-O_{d,B}(n^{-3}).
\end{align*}
Together with \eqref{eq:chainupper}, this proves
\eqref{eq:chainvalue}. All estimates depend only on $d$ and $B$,
not on the particular end subgraphs.
\end{proof}

\begin{proof}[Proof of \Cref{thm:minimum}]
Choose a graph attaining $m_d(n)$ in the indicated class.
\Cref{thm:GMfull} shows that, for all sufficiently large $n$,
it is an uninterrupted $L_d$ chain outside two end portions of
total order at most $B_d$. If necessary,
absorb the first and last $L_d$ copies into the corresponding end
portions. This increases their total order by at most $2D$ and
leaves a chain attached to two disjoint connected end subgraphs
by one bridge each.

Thus \Cref{lem:chainvalue}, with $B$ depending only on $d$, gives
\[
 m_d(n)=\frac{(d-1)\pi^2}{n^2}+O_d(n^{-3}).
\]
In the regular class the orders are even, as required. This proves
\eqref{eq:sharpminimum} as a consequence of the structural theorem.
\end{proof}

\section{Algebraic connectivity and maximum diameter}\label{sec:diameter}

In this section, we prove \Cref{cor:diameter}, establishing the
implication in \cite[Conjecture~5.4]{AbGhDiam} for the degree classes
stated in that theorem. We combine the stability theorem with the
classical diameter upper bound in \cite{Caccetta,Erdos} and
\cite[Theorems~5.1--5.2]{AbGhDiam}.

\begin{proof}[Proof of \Cref{cor:diameter}]
Equation~\eqref{eq:sharpminimum} gives
$n^2m_d(n)\to(d-1)\pi^2$. Hence $\mu(G_n)/m_d(n)\to1$ implies
$n^2\mu(G_n)\to(d-1)\pi^2$. Put
\[
 \e_n=\max\{0,n^2\mu(G_n)-(d-1)\pi^2\}.
\]
For all large $n$, $0\le\e_n\le1$ and $\e_n\to0$.
\Cref{thm:stability} gives
\[
 \dm(G_n)\ge\frac{3n}{d+1}
       -C_d(\e_n+n^{-1})^{1/3}n
       =\frac{3n}{d+1}-o(n).
\]
On the other hand, \cite{Caccetta,Erdos} and
\cite[Theorems~5.1--5.2]{AbGhDiam} give
$\dm(G_n)\le3n/(d+1)-1$. These two bounds prove
\eqref{eq:diameterimplication}.
For exact minimizers, \Cref{thm:GMfull} leaves $(n-O_d(1))/(d+1)$ uninterrupted
middle blocks in their respective classes. Their contribution to a
shortest path gives $\dm(G)\ge3n/(d+1)-O_d(1)$. Combine this with
the same cited diameter bound to obtain \eqref{eq:exactdiameter}.
\end{proof}

\section{Uniform minimum-degree bounds}
\label{sec:aldousfill}
In this section, we prove \Cref{thm:uniformminimum} by treating
sublinear minimum degree and minimum degree comparable to the order.
The Aldous--Fill bound then follows as a corollary.

\subsection{A finite-order bound}
We first deal with degrees smaller than the order. Extending the
comparison path by constant values at its two ends avoids having
to estimate the change in the norm caused by interpolation.

\begin{lemma}\label{lem:lowerbound}
Every connected  graph $G$ of order $n$ and minimum degree
at least $d\ge3$ satisfies
\begin{equation}\label{eq:simplelower}
 \mu(G)\ge(d-1)p_{n+2d},
\end{equation}
where $p_m=2(1-\cos(\pi/m))$. The integer $d$ need not be fixed
as $n$ varies.
\end{lemma}

\begin{proof}
Let $\x$ be a nondecreasing mean-zero unit Fiedler vector, and let
$\y$ be its interpolation from \Cref{lem:pathcomparison}.
Put $N=n+2d$. Extend $\y$ to $\z=(z_{1-d},\ldots,z_{n+d})^\top$
by adjoining $d$ copies of $x_1=y_1$ at the left and $d$ copies
of $x_n=y_n$ at the right, with $z_i=y_i$ for $1\le i\le n$.
The path energy is unchanged. Each interpolation interval has at
most $d+1$ differences. For each changed entry $x_i$, its unchanged
endpoints therefore lie between indices $i-d$ and $i+d$.
At every other index $z_i=x_i$, so monotonicity gives
\begin{equation}\label{eq:paddedbracket}
 z_{i-d}\le x_i\le z_{i+d}\qquad(1\le i\le n).
\end{equation}
Put $c=\bar z$, and let $k$ count the entries $x_i\le c$.
Compare $x_i$ with $z_{i-d}$ for $i\le k$ and with $z_{i+d}$
for $i>k$. The chosen indices are distinct, and each chosen
value is at least as far from $c$ as $x_i$. Hence
\begin{align*}
 \|\z-c\1\|^2
 &\ge\sum_{i=1}^{k}(z_{i-d}-c)^2
       +\sum_{i=k+1}^{n}(z_{i+d}-c)^2\\
 &\ge\sum_{i=1}^{n}(x_i-c)^2=1+nc^2\ge1.
\end{align*}
Since adjoining the endpoint copies adds no energy,
\eqref{eq:pathcomparison} and the path inequality give
\[
 \mu(G)\ge(d-1)\z^\top L(P_N)\z
 \ge(d-1)p_N\|\z-c\1\|^2\ge(d-1)p_N.
\]
\end{proof}

For fixed $d$, the right-hand side of \eqref{eq:simplelower} is
$(d-1)\pi^2/n^2+O_d(n^{-3})$. For the uniform argument, we use
the exact bound rather than this fixed-degree expansion.

\subsection{A weighted variance inequality}
The following inequality will be applied to a small graph obtained by
replacing each large part by one vertex. The weight of that vertex is
the proportion of original vertices in the part.

\begin{lemma}\label{lem:weightedvariance}
Let $Q$ be a connected  graph with $k\ge2$ vertices.
Give its vertices positive weights $w_1,\ldots,w_k$.
Put $W=\sum_iw_i$, $m=\min_iw_i$, and
$\bar z=W^{-1}\sum_iw_i z_i$. Then every real vector $\z=(z_1,\ldots,z_k)^\top$ satisfies
\begin{equation}\label{eq:weightedvariance}
 \sum_{i=1}^k w_i(z_i-\bar z)^2
 \le \frac{W^2}{8m}
       \sum_{ij\in E(Q)}(z_i-z_j)^2.
\end{equation}
The constant $8$ cannot be increased.
\end{lemma}

\begin{proof}
Relabel the vertices so that $z_1\le\cdots\le z_k$, retaining the weight
belonging to each vertex, and write $b_h=z_{h+1}-z_h$.
For an edge $ij$ with $i<j$, nonnegativity of the $b_h$ gives
\[
 (z_j-z_i)^2
 =\left(\sum_{h=i}^{j-1}b_h\right)^2
 \ge\sum_{h=i}^{j-1}b_h^2.
\]
Every cut between consecutive positions is crossed by an edge, because
$Q$ is connected. Summing the last inequality therefore gives
\begin{equation}\label{eq:weightedorderedpath}
 \sum_{ij\in E(Q)}(z_i-z_j)^2\ge\sum_{h=1}^{k-1}b_h^2.
\end{equation}
The weighted variance identity and Cauchy--Schwarz give
\begin{align}
 \sum_iw_i(z_i-\bar z)^2
 &=\frac1W\sum_{i<j}w_iw_j(z_j-z_i)^2\notag\\
 &\le\frac1W\sum_{h=1}^{k-1}b_h^2
      \sum_{i\le h<j}w_iw_j(j-i).
 \label{eq:variancepaths}
\end{align}
We bound the inner sum explicitly.

Partition the interval $[0,W]$ into consecutive intervals of lengths
$w_1,\ldots,w_k$, and let $s_i$ be the midpoint of interval $i$.
Since $s_{i+1}-s_i=(w_i+w_{i+1})/2\ge m$, we have
\[
 m(j-i)\le s_j-s_i.
\]
For a fixed $h$, put $A=\sum_{i\le h}w_i$ and $B=W-A$.
Integration of the coordinate over the two portions of $[0,W]$ gives
\begin{align*}
 \sum_{i\le h}w_i s_i&=\frac{A^2}{2},\\
 \sum_{j>h}w_j s_j&=\frac{W^2-A^2}{2}.
\end{align*}
Consequently,
\begin{align*}
 \sum_{i\le h<j}w_iw_j(j-i)
 &\le\frac1m\sum_{i\le h<j}w_iw_j(s_j-s_i)\\
 &=\frac1m\left(A\frac{W^2-A^2}{2}-B\frac{A^2}{2}\right)\\
 &=\frac{WAB}{2m}\\
 &\le\frac{W^3}{8m}.
\end{align*}
Substitution into \eqref{eq:variancepaths}, followed by
\eqref{eq:weightedorderedpath}, proves \eqref{eq:weightedvariance}.
For sharpness, take $Q=K_2$, $w_1=w_2=1$, and $(z_1,z_2)=(-1,1)$.
Then both sides of \eqref{eq:weightedvariance} are $2$.
\end{proof}

\subsection{Minimum degree comparable to the order}
Here the minimum degree is a positive proportion of the order. A vector
with small energy cannot change much across most incident edges. We keep
those edges and compare the resulting connected parts.

The diameter bound in \cite{Caccetta,Erdos} gives, for every
connected graph $H$,
\begin{equation}\label{eq:densediameter}
 \dm(H)<\frac{3|V(H)|}{\delta(H)+1}.
\end{equation}
This estimate is valid without fixing the minimum degree.

\begin{lemma}\label{lem:denseminimum}
Let $G_n$ be connected  graphs of orders $n\to\infty$, with
minimum degrees $\delta_n\ge c n$ for some fixed $c>0$.
Then
\[
 \liminf_{n\to\infty}\frac{n^2\mu(G_n)}{\delta_n}\ge8.
\]
Regularity is not required.
\end{lemma}

\begin{proof}
It suffices to consider an arbitrary subsequence on which
\[
 \frac{n^2\mu(G_n)}{\delta_n}\le C
\]
for some fixed $C$: any subsequence violating the asserted lower
limit would have this property.
Write $G=G_n$ and $\delta=\delta_n$.
Choose a mean-zero Fiedler vector $\y$ normalized by
$\sum_v y_v^2=n$. Its energy is
\begin{equation}\label{eq:boundedenergy}
 E:=\sum_{uv\in E(G)}(y_u-y_v)^2=n\mu(G)\le C,
\end{equation}
since $\delta\le n$.

Let $a_n=n^{-1/4}$. Form a spanning subgraph $H$ by retaining exactly
the edges for which $|y_u-y_v|\le a_n$.
At any vertex, each deleted edge contributes more than $a_n^2$ to
\eqref{eq:boundedenergy}. Thus at most $C/a_n^2=C\sqrt n$ incident
edges are deleted, and
\begin{equation}\label{eq:retaineddegree}
 \delta(H)\ge\delta-C\sqrt n.
\end{equation}
For sufficiently large $n$, this lower bound is at least $cn/2$.
Let $V_1,\ldots,V_k$ be the vertex sets of the connected components of
$H$, and write $n_i=|V_i|$. Each component has at least
$\delta-C\sqrt n+1$ vertices, so $k$ is bounded in terms of $c$.
By \eqref{eq:densediameter} and \eqref{eq:retaineddegree}, each component
also has diameter at most $6/c$. Every retained edge has value difference
at most $a_n$. It follows that
\begin{equation}\label{eq:oscillation}
 \max_{u,v\in V_i}|y_u-y_v|\le \omega_n,
\end{equation}
where $\omega_n=(6/c)n^{-1/4}$.

Put $w_i=n_i/n$, and let $z_i$ be the average of $\y$ on $V_i$.
The normalization of $\y$ gives
\begin{align}
 \sum_iw_i&=1,\notag\\
 \sum_iw_i z_i&=0,\notag\\
 \sum_iw_i z_i^2
 &=1-\frac1n\sum_i\sum_{v\in V_i}(y_v-z_i)^2\notag\\
 &\ge1-\omega_n^2.
 \label{eq:componentvariance}
\end{align}
In particular, $k\ge2$ for sufficiently large $n$: if $k=1$, its
average would be zero, contradicting the last inequality.
Moreover,
\begin{equation}\label{eq:minmass}
 \min_iw_i\ge\frac\delta n-\frac C{\sqrt n}.
\end{equation}

Contract the components of $H$ in $G$. The resulting graph is connected.
Choose a spanning tree $T$ on its $k$ vertices, and choose one original
edge $u_{ij}v_{ij}$ for each edge $ij$ of $T$, with its endpoints in
$V_i,V_j$, respectively. These original edges are distinct. By
\eqref{eq:oscillation},
\[
 \big|(z_i-z_j)-(y_{u_{ij}}-y_{v_{ij}})\big|\le2\omega_n.
\]
The triangle inequality for the Euclidean norm therefore gives
\begin{align*}
 \left(\sum_{ij\in E(T)}(z_i-z_j)^2\right)^{1/2}
 &\le
 \left(\sum_{ij\in E(T)}(y_{u_{ij}}-y_{v_{ij}})^2\right)^{1/2}
   +2\omega_n\sqrt{k-1}\\
 &\le\sqrt E+2\omega_n\sqrt{k-1}.
\end{align*}
This compares the values of one vector; no eigenvalue monotonicity under
contraction is assumed. Since $E$ and $k$ are bounded, squaring shows that
\begin{equation}\label{eq:treeenergy}
 \sum_{ij\in E(T)}(z_i-z_j)^2\le E+O_{c,C}(\omega_n).
\end{equation}
Apply \Cref{lem:weightedvariance} to $T$, with the weights $w_i$.
Equations \eqref{eq:componentvariance}--\eqref{eq:treeenergy} give
\begin{align*}
 E+O_{c,C}(\omega_n)
 &\ge8\min_iw_i\sum_iw_i z_i^2\\
 &\ge8\left(\frac\delta n-\frac C{\sqrt n}\right)(1-\omega_n^2).
\end{align*}
As $n/\delta\le1/c$ and $E=n\mu(G)$, multiplication by $n/\delta$
yields
\[
 \frac{n^2\mu(G)}\delta\ge8-o(1).
\]
This proves the lemma.
\end{proof}

\subsection{The uniform estimate and its sharpness}
We combine the finite-order bound with the estimate for minimum
degree comparable to the order. The matching examples are the
comparison chains when $r=3,4,5$, and two large cliques joined by
one edge when $r\ge6$.

\begin{proof}[Proof of \Cref{thm:uniformminimum}]
Suppose first that the uniform lower bound is false. Then, for
some $\varepsilon>0$, there are connected simple graphs $G_j$
of orders $n_j\to\infty$ and minimum degrees $\delta_j\ge r$ such
that
\[
 \frac{n_j^2\mu(G_j)}{\delta_j}<c_r-\varepsilon.
\]
Pass to a subsequence on which $\delta_j/n_j\to\rho\in[0,1]$.
If $\rho=0$, \Cref{lem:lowerbound}, with $d=\delta_j$, gives
\[
 \frac{n_j^2\mu(G_j)}{\delta_j}
 \ge\left(1-\frac1{\delta_j}\right)n_j^2p_{n_j+2\delta_j}.
\]
Since $(n_j+2\delta_j)/n_j\to1$, we have
$n_j^2p_{n_j+2\delta_j}\to\pi^2$. The lower limit is therefore
at least $(r-1)\pi^2/r\ge c_r$, a contradiction.
If $\rho>0$, then $\delta_j\ge(\rho/2)n_j$ for large $j$.
\Cref{lem:denseminimum} gives lower limit at least $8\ge c_r$,
again a contradiction. This proves the lower bound uniformly
over the minimum degree.

For sharpness when $r=3,4,5$, use the minimum-degree comparison
graphs from \Cref{lem:comparisonupper} with $d=r$. They exist at
every sufficiently large order, have minimum degree exactly $r$,
and satisfy
\[
 \frac{n^2\mu(G)}{\delta(G)}
 \le\frac{r-1}{r}\pi^2+O_r(n^{-1}).
\]
For $r\ge6$, let $a=\lfloor n/2\rfloor$ and $b=n-a$, and join
disjoint cliques $K_a,K_b$ by a single edge. For large $n$, its
minimum degree is $a-1\ge r$. Give every vertex of $K_a$ value
$b$ and every vertex of $K_b$ value $-a$. This vector has mean
zero, and only the joining edge contributes to its energy. Its
Rayleigh quotient is
\[
 \frac{(a+b)^2}{ab^2+ba^2}=\frac n{ab}.
\]
Consequently,
\[
 \frac{n^2\mu(G)}{\delta(G)}
 \le\frac{n^3}{ab(a-1)}\longrightarrow8.
\]
The lower bound and these constructions prove
\eqref{eq:uniformminimum}.
\end{proof}

The proof also gives a uniform conclusion for growing minimum
degree, including irregular graphs.

\begin{corollary}\label{prop:AFgrowing}
Let $G_j$ be connected simple graphs with orders $n_j\to\infty$
and minimum degrees $\delta_j\to\infty$. Then
\begin{equation}\label{eq:AF8}
 \liminf_{j\to\infty}\frac{n_j^2\mu(G_j)}{\delta_j}\ge8.
\end{equation}
If also $\delta_j=o(n_j)$, the lower bound improves to $\pi^2$.
\end{corollary}

\begin{proof}
When $\delta_j=o(n_j)$, apply \Cref{lem:lowerbound} with
$d=\delta_j$. The factor $1-1/\delta_j$ tends to one, and
$n_j^2p_{n_j+2\delta_j}\to\pi^2$.
For a general sequence, every subsequence on which
$\delta_j/n_j$ has a positive limit is covered by
\Cref{lem:denseminimum}. A subsequence with limit zero is covered
by the preceding argument. Passing to such a convergent
subsequence rules out a lower limit smaller than $8$.
\end{proof}

\subsection{The regular-graph conclusion}
For a $d$-regular graph, the transition matrix of the simple random
walk is $A/d$. Thus its spectral gap is $\mu(G)/d$, and its
relaxation time is $d/\mu(G)$. The minimum-degree theorem now
gives the uniform regular-graph bound.

\begin{proof}[Proof of \Cref{thm:aldousfill}]
For $d\ge3$, apply \Cref{thm:uniformminimum} with $r=3$ and use
$\delta(G)=d$. The lower bound is uniform over all these degrees.
For $d=2$, connectedness gives $G=C_n$, and
\[
 \frac{n^2\mu(C_n)}2
 =n^2\bigl(1-\cos(2\pi/n)\bigr)\longrightarrow2\pi^2,
\]
which is larger than $2\pi^2/3$.
Finally, the cubic comparison graphs from
\Cref{lem:comparisonupper} exist at every sufficiently large even
order and satisfy $\mu(G)\le2\pi^2/n^2+O(n^{-3})$.
Together with the lower bound, this proves sharpness along even
orders.
\end{proof}

\end{document}